\documentclass[10pt]{article}

\usepackage[a4paper]{geometry}
\usepackage{graphicx}
\usepackage{amsmath}
\usepackage{authblk}
\usepackage[english]{babel}
\usepackage{amsmath}
\usepackage{amsfonts}
\usepackage{stmaryrd}
\usepackage{amssymb}
\usepackage[usenames,dvipsnames]{color}
\usepackage[colorlinks=true]{hyperref}
\usepackage[capitalize,nameinlink,noabbrev]{cleveref} 
\usepackage{setspace}
\usepackage{dsfont}
\usepackage[all,cmtip]{xy}
\usepackage{amsthm}
\usepackage{mathrsfs}
\usepackage[square,numbers]{natbib}
\numberwithin{equation}{section}

\newtheorem{theorem}{Theorem}[section]

\newtheorem{lemma}[theorem]{Lemma}
\newtheorem{proposition}[theorem]{Proposition}

\theoremstyle{definition}
\newtheorem{definition}[theorem]{Definition} 
\newtheorem{example}[theorem]{Example}
\newtheorem{observation}[theorem]{Observation}
\newtheorem{nolabel}[theorem]{}
\newtheorem*{prop*}{Proposition}
\newtheorem*{thm*}{Theorem}

\newcommand{\Ber}{\mathrm{Ber}}
\newcommand{\spec}{\mathrm{Spec}}
\newcommand{\sdiv}{\mathrm{sdiv}}
\newcommand{\Aut}{\mathrm{Aut}}
\newcommand{\Der}{\mathrm{Der}}
\newcommand{\car}{\mathrm{char}}

\newcommand{\Ad}{\mathrm{Ad}}
\newcommand{\rd}{\mathrm{rd}}

\def\cL{\mathcal{L}}
\def\cT{\mathcal{T}}
\def\cM{\mathcal{M}}
\def\cO{\mathcal{O}}
\def\cE{\mathcal{E}}
\def\cA{\mathcal{A}}
\DeclareMathOperator{\Sym}{Sym}

\begin{document}
	
\onehalfspacing	
	
	\title{On super curves and supervolumes}
	\author[1]{Ricardo Jes\'us Ramos Castillo\thanks{Email: rjesus@impa.br}}
	\date{\today}
	\maketitle
	
\begin{abstract}
We study the geometry of super curves  with a chosen supervolume form. We consider the algebra of divergence free vector fields $S(1|N)$ associated to such curves. When $N=2$ its derived algebra, called $S(2)$, defines a special family of curves, named $S(2)$-super curves. We exhibit an involution on the moduli space of such curves that generalizes Deligne's involution for $N=1$ super curves. The fixed point set of this involution consists on Manin's $SUSY_2$-super curves. We describe the moduli spaces of these curves. 
\end{abstract}

\section{Introduction}
\begin{nolabel}
In his famous work on classification of Lie super algebras \cite{kac-advances}, Kac introduced a list of infinite dimensional Lie super algebras generalizing the ordinary theory of Lie algebras of Cartan type. These are subalgebras of derivations on a super-commutative algebra freely generated by $n$ even variables and $N$ odd variables. In the case $n=1$ these are algebras of vector fields on a punctured super-disc $\spec \mathbb{C}(( t))[\theta_1,\dots,\theta_N]$ preserving some extra structure. In this case, the list consists of 
\begin{enumerate}
\item $W(1|N)$, all vector fields.
\item $S(1|N)$, divergence free vector fields, that is, vector fields acting trivially on the Berezininan, or preserving the section $[dt|d\theta_1\dots d\theta_N]$ of the Berezinian bundle. 
\item $CS(1|N)$, vector fields that preserve the Berezinian up to multiplication by a scalar function. 
\item $K(1|N)$, vector fields preserving a \emph{contact-like} form 
\[ dt + \sum \theta_i d\theta_i, \]
up to multiplication by a scalar function. 
\end{enumerate}
Some of these algebras are isomorphic, for example $W(1|1) \simeq K(1|2)$. While some others are not simple, for example $S(1|2)$ is not simple, but its derived algebra $S(2)$ is. 

For each such Lie super algebra, there is an associated class of super curves, with certain geometric structures preserved by these vector fields. That is, the class of super curves admitting an et\'ale cover where infinitesimal changes of coordinates are given by vector fields in the corresponding Lie super algebra. For example, a $W(1|1)$-super curve is a general $1|1$-dimensional super curve, they consist of a smooth curve $C$ together with a line bundle $\cL$ over it. $K(1|1)$-super curves are called $SUSY$-super curves by Manin in \cite{manin2014topics} and $SUSY$-super Riemman surfaces in \cite{FalquiReina,rabin:duke}, they consist of a smooth curve $C$ and a choice of a square root of the canonical bundle $\omega_C$.  Similarly $K(1|2)$-super curves are called (oriented) $SUSY_2$-super curves by Manin. In  \cite{Vaintrob} Vaintrob  studied the geometry of all these super curves, obtaining a description of the corresponding moduli spaces in each case.  

In this article we focus on one example that is missing in Vaintrob's list, these are the so-called $S(2)$-super curves. These are smooth $1|2$-dimensional super curves, endowed with a trivializing section of its Berezinian bundle and with the additional condition that the above mentioned changes of coordinates lie in Kac's $S(2)$-algebra as opposed to the full algebra $S(1|2)$. 
\end{nolabel}

\begin{nolabel}
Deligne exploited the isomorphism $W(1|1) \simeq K(1|2)$ in \cite{deligne} to describe an involution in the moduli space of general smooth $1|1$-dimensional super curves, the fixed locus of which is the moduli space of $K(1|1)$-super curves. This involution is induced by an involution of the Lie super algebra $K(1|2)$, fixing its subalgebra $K(1|1)$. Geometrically, a $W(1|1)$-super curve, or a general $1|1$-dimensional super curve, over a purely even super scheme $S$ (that is simply a scheme $S$) is given by a smooth curve $C$ over $S$ together with a line bundle $\cL$ over it. Deligne's involution corresponds to taking the Serre dual of $\cL$: \[ (C, \cL) \leftrightarrow (C, \omega_{C/S} \otimes_{\cO_{C}} \cL^*).\] The fixed point set of this involution is parametrized by curves $C$ together with a choice of a \emph{theta-characteristic} (a square root of the canonical bundle). This is the moduli space of $K(1|1)$-super curves as shown in \cite{Vaintrob}. 

In \cite{Donagi:2014hza}, Donagi and Witten show that when the base $S$ is an arbitrary (non-necessarily even) super scheme there exist non-split $W(1|1)$ super curves over $S$. In particular these curves are not given as the spectrum of the free super commutative algebra generated by a line bundle $\cL$ as above. The description of Deligne's involution in this case is not so transparent.
\label{no:1.2}
\end{nolabel}
\begin{nolabel}Our main result is to generalize Deligne's involution to the case of $S(2)$ curves. In order to  do so  we consider the sequence of inclusions
\begin{equation}  K(1|2) \subset S(2) \subset S(1|2) \subset K(1|4). 
\label{eq:inclusions}
\end{equation} 
This provides a sequence of embeddings of the corresponding moduli spaces: each $K(1|2)$-super curve comes with a trivialization of its Berezinian bundle and the local changes of coordinates are in $S(2)$. Each $S(2)$-super curve is in particular a $1|2$-dimensional super curve with a trivialization of its Berezinian bundle $\Ber_C$ (an $S(1|2)$-super curve). Similarly if $C$ is a $1|2$-dimensional super curve with a trivialization of $\Ber_C$, consider its tangent bundle $\cT_C$, a locally free $\cO_C$ module of rank $1|2$. The Grassmanian $\widetilde{C}$ of rank $0|2$ subbundles of $\cT_C$ is a $1|4$-dimensional super curve with a canonical $K(1|4)$-structure \cite{dolgikh1990}. We show that there exists an involution of $K(1|4)$ that fixes pointwise its subalgebra $K(1|2)$ and preserves (but does not fix) $S(2)$. This involution implies

\begin{theorem}[Theorem \ref{teorema de automorfismo}] 
	There exists an involution $\mu$ of the moduli space $\cM_{S(2)}$ of $S(2)$-super curves. The fixed point set  of $\mu$ consists of the moduli space $\cM_{K(1|2)}$ of orientable $SUSY_2$-super curves. 
\end{theorem}

There are super curves with trivial Berezinian that are not $S(2)$-super curves. And the Lie super algebra $S(1|2)$ is not stable under the involution $\mu$ above. This shows that our generalization of Deligne's involution requires precisely the $S(2)$-structure as opposed to $S(1|2)$. 
\label{no:princ-1}
\end{nolabel}
\begin{nolabel}
$S(1|2)$ curves admit a simple geometrical description: these are $1|2$-dimensional super-curves together with a trivialization of its Berezinian bundle. In contrast, $S(2)$ curves do not admit this simple geometrical description. To any $S(1|2)$ curve $C$ we attach an affine bundle, or a $\mathbb{G}_a$ torsor $K_C$. The class of this bundle is an obstruction for the $S(1|2)$ curve to be an $S(2)$ curve, namely $C$ is an $S(2)$ curve if and only if $K_C$ is trivial (Proposition \ref{prop:trivialga}).
\end{nolabel}

\begin{nolabel}Given a $W(1|1)$-super curve $C$  over a purely even super scheme $S$, it is split in the sense that there exits a smooth $1|0$-dimensional curve $C_0$ over $S$ and a line bundle $\cL$ over $C_0$ such that $C = \spec \Sym_{\cO_{C_0}} \cL[-1]$. This allowed us to describe Deligne's involution as taking Serre's dual. A similar situation arises in the $S(2)$ case: for a purely even scheme $S$ and an $S(2)$-super curve $C$ over $S$, there exists a smooth $1|0$-dimensional curve $C_0$ over $S$, a rank two bundle $\cE$ over $C_0$ satisfying $\det \cE \simeq \omega_{C_0/S}$ and such that $C \simeq \spec \Sym_{\cO_{C_0}} \cE[-1]$. That is we have
	
\begin{theorem}[Theorem \ref{split}] 
	Every $S(2)$-super curve over a purely even base $S$ is split. 
\label{thm:split1}
\end{theorem}

In this situation, our involution above is given just as in the $W(1|1)$ case: it corresponds to \[ (C_0, \cE) \leftrightarrow (C_0, \omega_{C_0/S} \otimes_{\cO_{C_0}} \cE^*).\]

\label{no:noins}
\end{nolabel}
\begin{observation}
Theorem \ref{thm:split1} is false for $S(1|2)$ super curves. Even over a purely even base $S$, there are $S(1|2)$-super curves that are not split (see Example \ref{ex:non-split12}), and therefore they are not $S(2)$-super curves. 

The condition on the base $S$ on Theorem \ref{thm:split1} is necessary, that is, there exists families of $S(2)$-curves over super schemes that are not split (see Example \ref{familia S(2) no split}).

From this point of view, the condition on a super curve with trivial Berezinian, of being an $S(2)$-super curve is the analog of the condition of being oriented for general $SUSY_2$-super curves as in \cite{manin2014topics}. 
\end{observation}
\begin{nolabel}Our second result is a description of the moduli spaces of $S(2)$ and $S(1|2)$ supercurves. We first characterize the universal family of such curves over a purely even base $S$:
\begin{prop*}(See Proposition \ref{even data} below) The data of a family of $S(2)$-supercurves $C \rightarrow S$ whose reduction coincides with a given family of curves $C_0 \rightarrow S$ over a purely even scheme $S$ is equivalent to a rank $2$ vector bundle $E \rightarrow C_0$ together with an isomorphism $\det E \xrightarrow{\beta} \Omega_{C_0/S}$. Two such supercurves $(E, \beta)$ and $(E', \beta')$ are equivalent if and only if there exists a bundle isomorphism $\alpha: E \rightarrow E'$ such that $\beta' \circ \det \alpha = \beta$. 
\end{prop*}
Similarly for $S(1|2)$ supercurves we have Proposition \ref{even data S(1|2)}:
\begin{prop*}
A family of $S(1|2)$ curves $C \rightarrow S$ with a given reduction $\pi: C_0 \rightarrow S$ over a purely even base $S$ is determined by a rank $2$ bundle $E \rightarrow C_0$ together with an isomorphism $\beta: \det E \rightarrow \Omega_{C_0/S}$ and a class $\Gamma \in H^1(C_0, \pi^* \mathcal{O}_S)$. Two such supercurves $(E, \beta, \Gamma)$ and $(E', \beta', \Gamma')$ are equivalent if $\Gamma= \Gamma'$ and the pairs $(E,\beta)$ and $(E', \beta')$ are equivalent as in the previous proposition. 
\end{prop*}
The map $(E, \beta, \Gamma) \rightarrow (E,\beta)$ could be thought of as a fibration from the moduli space of $S(1|2)$ supercurves to the moduli space of $S(2)$ curves over purely even bases. However, there are non-trivial odd deformations of $S(2)$ supercurves. 

We describe the full moduli space of $S(2)$ supercurves under the assumption that the base super-scheme is split. Given such a split super-scheme $S$ with purely even reduction $S_{\mathrm{rd}}$. The datum of a family of $S(2)$ curves $C \rightarrow S$ with reduction $C_0 \rightarrow S_{\mathrm{rd}}$ is equivalent to a class in $H^1(C_0,G)$ where $G$ is a sheaf of groups over $C_0$ described in \ref{group bundle} (see Proposition \ref{odd part}).

\label{no:desc-moduli}
\end{nolabel}

\begin{nolabel}The organization of this article is as follows. In section \ref{sec:super-algebras} we recall the basic preliminaries on super-commutative algebras, their modules, and their derivations. We introduce the relevant infinite dimensional Lie algebras in Kac's list and describe their associated infinite dimensional groups as groups of automorphisms of a super-disk preserving certain geometric structure. 

In section \ref{sec:super-manifolds} we recall the basic preliminaries on super-geometry. Define our curves of interest and construct a principal $\mathbb{G}_a$ bundle characterizing the obstruction of an $1|2$ dimensional super-curve with trivial Berezinian being an $S(2)$ super-curve. We show that every $S(2)$ curve over a purely even base is split (Theorem \ref{split}) and we finish that section attaching a SUSY $1|2N$ curve to any $1|N$ super-curve. 

In section \ref{sec:duality} we describe the involution of the moduli space of $S(2)$ supercurves generalizing that of Deligne for general $1|1$ supercurves. 
In section \ref{sec:families} we give the above mentioned examples of families of supercurves: $S(1|2)$ supercurves that are not $S(2)$ supercurves. Non-split $S(1|2)$ supercurves over a purely even base. Non-split $S(2)$ supercurves over a super-scheme. 
In section \ref{sec:moduli} we give a description of the moduli spaces of $S(1|2)$ and $S(2)$ supercurves. We describe the subspace of supercurves over purely even schemes and then describe the possible deformations of such a curve in the {\em odd directions of the base} under the assumption that the base is a split superscheme.

We identify the full automorphism group of such families of supercurves for genus $g \geq 2$ and describe the corresponding orbifold quotient.
\label{no:organization}
\end{nolabel}
\section{Super algebra}
 \label{sec:super-algebras}

\subsection{Super algebras}

\begin{definition}
	A \emph{$k$-super algebra $R$ over a field $k$} is a $\mathbb Z/2\mathbb Z$-graded $k$-vector space, $R=R_0\oplus R_1
	$\footnote{For super spaces we will write $R_i$ meaning the $i$-th part of $R$ for $i=0,1\mod 2$.} with a unital multiplication $R\otimes_k R\to R$ that respects the gradation, i.e. $R_iR_j\subset R_{i+j}$, such that for homogeneous elements $a\in R_i$, $b\in R_j$ we have the commutative rule: $ba=(-1)^{ij}ab$.
	
	Let $R,S$ be $k$-super algebras, a $k$-linear map $T:R\to S$ is said to be \emph{even} if $T(R_i)\subset S_i$, $i=0,1\mod 2$, and is said to be \emph{odd} if $T(R_i)\subset S_{i+1}$, $i=0,1\mod 2$. An even $k$-linear map $T:R\to S$ is said to be a \emph{homomorphism of super algebras} if $T(rr')=T(r)T(r')$ for any $r,r'\in R$ and $T(1)=1$. The set of super algebras homomorphisms is going to be denoted by $Hom_{SAlg}(R,S)$. Let $R,S$ be super algebras over $k$, we say that \emph{$R$ is an $S$ algebra} if there exists a homomorphism of super algebras $S\to R$.
	
	 An element $a\in R$ is called \emph{even} if $a\in R_0$ and is said to be \emph{odd} if $a\in R_1$. Also, we say that $a\in R$ has \emph{parity $j$} if $a\in R_j$. 
	
	For a non-nilpotent even element $f\in R_0$ we denote by $R_f$ the super algebra given by the localization of $R$ with respect to the multiplicative set $\{1,f,f^2,\dots\}$.

\end{definition}

\begin{example}
	A commutative ring $R$ over $k$ can be seen as a super algebra with $R_0=R$ and $R_1=0$. Also, observe that for any super algebra $R$, $R_0$ is a commutative ring.
\end{example}

\begin{example}\label{ejemplo}
	Given a super algebra $R$, the super algebra of polynomials $R[t]$, with $t$ an even variable, is defined as the usual algebra of polynomials with the $\mathbb Z/2\mathbb Z$-gradation:
		\begin{equation*}
			R[t]_j:=\{a_0+a_1t+\dots a_nt^n:n\in\mathbb N,a_i\in R_j\},	
		\end{equation*}
	then, $R[t]$ is a super algebra. Recursively, we will consider the super algebra $R[t_1,\dots,t_n]:=R[t_1,\dots,t_{n-1}][t_n]$, for the even variables $t_1,\dots, t_n$.
	
	For a super algebra $R$ we can construct the \emph{Grassmann algebra} $R[\theta]$, with $\theta$ an odd variable, defined as the usual algebra of polynomials with the $\mathbb Z/2\mathbb Z$-gradation:
		\begin{equation*}
			R[\theta]_j:=\{a_0+a_1\theta:a_0\in R_j,a_1\in R_{j+1}\},	
		\end{equation*}
	then, $R[\theta]$ is a super algebra with $\theta\in R[\theta]_1$. Recursively, we will consider the Grassmann algebra of rank $n$, $R[\theta^1,\dots,\theta^n]:=R[\theta^1,\dots,\theta^{n-1}][\theta^n]$, for the odd variables $\theta^1,\dots,\theta^n$.
	
	We write $R[m|n]:=R[t_1,\dots,t_m][\theta^1,\dots,\theta^n]$, where $t_1,\dots,t_m$ are even and $\theta^1,\dots,\theta^n$ are odd variables.
	
	Similarly, we define $R[[ m|n]]:=R[[ t_1,\dots,t_m]] [\theta^1,\dots,\theta^n]$, for the even variables $t_1,\dots,t_m$ and odd variables $\theta^1,\dots,\theta^n$.
\end{example}

\begin{definition}
	Let $R$ be a super algebra and put $J:=R_1+R_1^2$.	We define the \emph{reduced ring of $R$} as the quotient $R_{\text{rd}}:=\frac{R}{J}$, this is a ring endowed with the projection $R\to R_{\text{rd}}$.
\end{definition}

\begin{observation}
	Let us consider the category of super algebras over a field $k$, $SSch/k$. Over it, any super algebra $R$ define a functor
		\begin{equation*}
			\begin{split}
			h_R:SAlg\to&Sets\\
			S\mapsto&Hom_{SAlg}(R,S).
			\end{split}
		\end{equation*}
	If we consider a commutative $k$-algebra $S$, then any morphism $R\to S$ vanishes in $R_1$, since $S_1=0$. Then such morphism factorizes through the projection $R\to R_{\text{rd}}$, so we get the natural bijection
		\begin{equation*}
			Hom_{SAlg}(R,S)\to Hom_{CAlg}(R_{\text{rd}},S),
		\end{equation*}
where $CAlg$ is the category of commutative $k$-algebras. 
	Finally, we get the lemma:
	
	\begin{lemma}\label{caracterizacion de reduccion}
		The restriction $h_R|_{CAlg}$ is represented by $R_{\text{rd}}$.
	\end{lemma}	 
\end{observation}

\begin{example}
	For a commutative algebra $R$ and the Grassmann algebra $R[\theta^1,\dots,\theta^n]$, observe that
		\begin{equation*}
			R[\theta^1,\dots,\theta^n]\to(R[\theta^1,\dots,\theta^n])_{\text{rd}}\simeq R,
		\end{equation*}
	and there exists a section $R\to R[\theta^1,\dots,\theta^n]$.
\end{example}
	
\subsection{Super modules}
	
\begin{definition}
	Let $R$ be a $k$-super algebra and consider $M$ be $\mathbb Z/2\mathbb Z$-graded $k$-vector space, we will say that $M$ is a \emph{module over $R$} if is endowed with a $k$-bilinear product
		\begin{equation*}
			R\otimes_k M\to M
		\end{equation*}
	such that $R_{i}M_{j}\subset M_{i+j}$, and for any $a,b\in R$ and $m\in M$:
		\begin{equation*}
			\begin{split}
				a\cdot(b\cdot m)=&\ (ab)\cdot m,\\
				1\cdot m=&\ m.
			\end{split}
		\end{equation*}
		
	For a super module $M$ we can construct the super module $\Pi M$ by
		\begin{equation*}
			\Pi M_{i}:=M_{i+1}.
		\end{equation*}
\end{definition}

\begin{example}
	For a super algebra $R$ the super algebra $R[n|m]$ is a module over $R$. In particular, $R$ is a module over $R$.
	
	The direct sum $R^{m|n}:=R^m\oplus (\Pi R)^n$ is also a module over $R$.
\end{example}

\begin{definition}
	For two super modules $M_1$, $M_2$ over $R$ and a $k$-linear map $T:M_1\to M_2$, we say that $T$ is \emph{even} if $T(M_{1,i})\subset M_{2,i}$, $i=0,1$, and \emph{odd} if $T(M_{1,i})\subset M_{2,i+1}$, $i=0,1$. We say that $T$ is \emph{an homogeneous $R$-homomorphism of $R$ modules} if $T$ has parity $j$ and for any $a\in R_i$ we have $T(am)=(-1)^{ij}aT(m)$.
		
	The space of $R$-homomorphism of $R$ modules, denoted by $\underline{\mathrm{Hom}}_R(M_1,M_2)$, is the $\mathbb Z/2\mathbb Z$-graded space of even and odd $R$-homomorphism:
		\begin{equation*}
			\underline{\mathrm{Hom}}_R(M_1,M_2):=\mathrm{Hom}_R(M_1,M_2)_{0}\oplus\mathrm{Hom}_R(M_1,M_2)_{1}.
		\end{equation*}
		
	We say that $T\in \underline{\mathrm{Hom}}_R(M_1,M_2)$ is \emph{invertible} if there exists an $S\in\underline{\mathrm{Hom}}_R(M_2,M_1)$ such that $T\circ S=\mathrm{id}_{M_2}$ and $S\circ T=\mathrm{id}_{M_1}$. An even element $T\in \underline{\mathrm{Hom}}_R(M_1,M_2)$ that is invertible is called \emph{isomorphism}. In this case, we say that $T$ has \emph{inverse} $S$ and that $M_1$ and $M_2$ are \emph{isomorphic}. When $M_1=M_2$, an isomorphism $T:M_1\to M_1$ is called \emph{automorphism} instead of isomorphism.
	
	An $R$-module $M$ is \emph{free, finitely generated and has rank $m|n$} if $M$ is isomorphic to $R^{m|n}$.
\end{definition}

\begin{example}
	Let $R$ be a super algebra, for a super module $M$, an element $a\in R_i$ induces an $R$-homomorphism with parity $i$ by multiplication:
	\begin{equation*}
	\begin{split}
	T_a:M\to&\ M\\
	m\mapsto&\ am.
	\end{split}
	\end{equation*}
	Observe that any invertible element in $R_0$ induce an automorphism in $M$.
\end{example}

From the odd $k$-linear map $M\to \Pi M$, $m\to m$, we obtain $M= \Pi\Pi M$. Note, that $M\to \Pi M$ is bijective but not an isomorphism since this morphism is odd.

\begin{observation}
	Let $M$ be a super module, and take
	\begin{equation*}
	\mathrm{End}_R(M):=\underline{\mathrm{Hom}}_R(M,M),
	\end{equation*}
	that is a $\mathbb Z/2\mathbb Z$-graded algebra with the composition as product. The set of invertible automorphisms is going to be denoted by $\mathrm{End}_R(M)^*$.
\end{observation}

\begin{observation}\label{reduction}
	Let $M$ be a free super module of rank $m|n$, i.e. $M$ is isomorphic to $R^m\oplus(\Pi R)^n$, observe that 
		\begin{equation*}
			M_{\text{rd}}:=R_{\text{rd}}\otimes_{R}M
		\end{equation*}
	is a rank $m|n$ free module over $R_{\text{rd}}$.
	
	Observe that if $M$ has rank $m|n$, then $\Pi M$ has rank $n|m$.
\end{observation}

\begin{example}\label{super algebras a traves de modulos}
	For a finitely generated free super module $M$ of rank $m|n$ over $R$ with generators $\{t_1,\dots,t_m|\theta^1,\dots,\theta^n\}$, with $t_i$ even and $\theta^j$ odd, the construction given in \ref{ejemplo} gives us a super algebra, $S_R[M]:=R[t_1,\dots,t_m|\theta^1,\dots,\theta^n]$. Observe, that this super algebra does not depend on the choice of generators.
\end{example}
	
\begin{definition}
	Let $M$ be a free super module of rank $m|n$ and $T:M\to M$ be an  invertible morphism represented in some basis by the matrix
		\begin{equation*}
			T=\left[\begin{matrix}A & B\\C &D \end{matrix}\right]
		\end{equation*}
	where $A,B,C,D$ is a $m\times m$, $m\times n$, $n\times m$ and $n\times n$ matrix, respectively, we define the \emph{Berezinian of $T$} by
		\begin{equation}\label{bere}
			\Ber(T)=\det(A-BD^{-1}C)\det(D)^{-1}.
		\end{equation}
\end{definition}
		
\begin{observation}
	The Berezinian verifies the following conditions:
		\begin{enumerate}
			\item If $T=\left[\begin{matrix}A & B\\0&D\end{matrix}\right]$ or $T=\left[\begin{matrix}A & 0\\C&D\end{matrix}\right]$, then $\Ber(T)=\det(A)\det(D)^{-1}$.
			\item Let $T,S$ be two automorphisms, then $\Ber(TS)=\Ber(T)\Ber(S)$. In particular, $\Ber(T)$ does not depend on the basis chosen.
			\item Suppose that $k=\mathbb C$ and that $M$ is finitely generated, so for $T\in \mathrm{End}_k(M)$ we can define $\exp(T)=\sum_{i\geq0}\frac{T^i}{i!}$. In this case we have
				\begin{equation*}
					\Ber(\exp(T))=\exp(\text{str}(T)),
				\end{equation*}
			where $\text{str}(T):=\text{tr}(A)-\text{tr}(D)$ is the \emph{super-trace} of $T$. 
			\item To define the Berezinian, we just need that $D$ in \eqref{bere} is invertible, and the observations above still hold even when $T$ is not necessarily invertible.
		\end{enumerate}
\end{observation}

\begin{observation}\label{modulo bereziniano} 
	For the free super module $M$ with generators $\{t_1,\dots,t_n,\theta^1,\dots,\theta^m\}$, with $t_i$ even and $\theta^j$ odd, we can construct the free module $\Ber(M)$ generated by the formal element $[t_1,\dots,t_n|\theta^1,\dots,\theta^m]$ with parity $m\mod 2$. Then $\Ber(M)$ has rank $1|0$ if $m$ is even and rank $0|1$ if $m$ is odd. An invertible homomorphism $T:M\to M$, induce the automorphism $\Ber(T):\Ber(M)\to\Ber(M)$. Finally, we get an homomorphism of super groups:
	\begin{equation*}
	\begin{split}
	End_R(M)^*\to&\ End_R(\Ber(M))^*\\
	T\mapsto&\ \Ber(T).
	\end{split}
	\end{equation*}
	\end{observation}
	
\subsection{Super derivations}

\begin{definition}
	Let $R$ be a $S$-super algebra and let $D\in End_S(R)$ with parity $i$. We say that $D$ is \emph{an $S$-derivation} if $D(s)=0$ for any $s\in S$, and for any $a\in R_j$ and $b\in R$ we have:
		\begin{equation*}
			D(ab)=D(a)b+(-1)^{ij}aD(b).
		\end{equation*}
		
	The vector space of derivations has a structure of $R$ super module given by $	(aD)(b)=aD(b),\text{ for any }a,b\in R$. We denote by $\mathrm{Der}_{R/S}$ the \emph{super module of derivations}.
	
	For two derivations $D_1\in \mathrm{Der}_{R/S,i}$ and $D_2\in \mathrm{Der}_{R/S,j}$ we define \emph{the bracket} by:
		\begin{equation*}
			[D_1,D_2]=D_1D_2-(-1)^{ij}D_2D_1,
		\end{equation*}
	with this structure $\mathrm{Der}_{R/S}$ is a super Lie algebra.
\end{definition}
	
\begin{example}
	Let $R[m|n]:=R[t_1,\dots,t_m][\theta^1,\dots,\theta^n]$, where $t_1,\dots,t_m$ are even and $\theta^1,\dots,\theta^n$ are odd, be the super algebra of polynomials associated to the super algebra $R$, then the set of derivations of $R[m|n]$ over $R$ is a $R[m|n]$ free module with even part generated by $\{\partial_{t_1},\dots,\partial_{t_m}\}$ and odd part generated by $\{\partial_{\theta^1},\dots,\partial_{\theta^n}\}$.
	
	Similarly, the generators of $\Der_{R}(R[[ 1|n]])$ are given by $\{\partial_{t_1},\dots,\partial_{t_m}|\partial_{\theta^1},\dots,\partial_{\theta^n}\}$.
\end{example}

For a vector field $X=A_0\partial_z+A_1\partial_{\theta^1}+\cdots+A_N\partial_{\theta^N}\in \Der_{R}(R[[ 1|n]])$, we define \emph{the super divergence operator}:
	\begin{equation*}
		\sdiv(X):=\partial_zA_0+(-1)^{A_1}\partial_{\theta^1}A_1+\cdots+ (-1)^{A_N}\partial_{\theta^N}A_N,
	\end{equation*}
and denote by $S(1|N)$ the \emph{space of divergence free vector fields}. 

\begin{observation}\label{generators}
	Let $R[[1|2]]$ be a super algebra with generators $\{t|\theta^1,\theta^2\}$ and maximal ideal $\mathfrak{m}=\langle t|\theta^1,\theta^2\rangle$. The space of divergence free vector fields in $\Der_{R}(R[[ 1|2]]_{\mathfrak{m}})$ has even generators:
	\begin{equation*}
	\begin{split}
	L_m  = &-z^{m+1}\partial_z-\frac{m+1}{2}z^m\sum_{i=1}^N \theta^i\partial_{\theta^i},\text{ for $m\in \mathbb Z$}, \\
	J_m^{0} = & z^m(\theta^1\partial_{\theta^1}- \theta^2 \partial_{\theta^2}),\text{ for $m\in \mathbb Z$}, \\
	J_m^{1} = & z^m\theta^1\partial_{\theta^2},\text{ for $m\in \mathbb Z$}, \\
	J_m^{2} = & z^m\theta^2\partial_{\theta^1},\text{ for $m\in \mathbb Z$}, \\
	K = & \theta^1\theta^2\partial_z,
	\end{split}
	\end{equation*}
	and odd part with generators:
	\begin{equation*}
	\begin{split}
	G_m^i = & -z^{m+1/2}\partial_{\theta^1},\text{ for $m\in \mathbb Z$}, \\
	H_m^i = & z^{m+1/2}\theta^i\partial_z -\left(m+1/2\right)z^{m-1/2}\theta^i\sum_{j=1}^N \theta^j\partial_{\theta^j},\text{ for $m\in \mathbb Z$}, .
	\end{split}
	\end{equation*}	
	Then, the algebra $S(1|2)$ is not simple, but its derived algebra 
	\begin{equation*}
	S(2):=[S(1|2),S(1|2)]
	\end{equation*}
	is simple (this is shown in \cite{article}). Observe that $S(1|2)/S(2)$ is a rank 1 Lie algebra. The map $S(1|2)\to S(1|2)/S(2)$	has its image generated by $K=\theta^1\theta^2\partial_z$. For $X\in S(1|2)$, then $X\in S(2)$ if and only if $\partial_{\theta^1}\partial_{\theta^2}(X\cdot z)=0$.
\end{observation}

\subsection{The group $\Aut_R(R[[ m|n]])$}

Let $R$ be a super algebra, for the super algebra $R[[ m|n]]$ we will consider the collection of automorphisms of $R$-super algebras and denote this group as $\Aut_R(R[[ m|n]])$, that is the set of even maps $T:R[[ m|n]]\to R[[ m|n]]$ such that $T|_R=id_R$. If there is no confusion, we will write $\Aut(R[[ m|n]])$ or $\Aut[[ m|n]]$. 

We are interested in the group of automorphism $\Aut (R[[ 1|n]])$ and its group structure given by $\Phi*\Psi=\Psi\circ\Phi$.

\begin{example}\label{exponencial}
	Let $R$ be a $k$ super algebra, $\car(k)=0$, for a nilpotent $X\in\Der_R(R[[ 1|n]])_0$, we define its exponential by:
		\begin{equation}\label{formula de exponencial}
			\exp(X)=\mathrm{id}+\frac{X}{1!}+\frac{X^2}{2!}+\cdots.
		\end{equation}
	For two nilpotent $N,L\in\Der_R(R[[ 1|n]])_0$ with $[N,L]=0$ we get 
		\begin{equation*}
			\exp(N+L)=\exp(N)\exp(L).
		\end{equation*}
	In particular, $\exp(N)$ has an inverse $\exp(-N)$. Finally, $\exp(X)\in\Aut_R(R[[ m|n]])$.
	
\end{example}

Let $m\in\mathbb N$, and the $R$-super algebra $R[t|\theta^1,\cdots,\theta^n]/\mathfrak{m}^{m}$, where $\mathfrak{m}:=\langle t|\theta^1,\dots,\theta^n\rangle$. Similar to Example \ref{exponencial}, an element $X\in\Der_{R}(R[t|\theta^1,\cdots,\theta^n]/\mathfrak{m}^{m})$ is nilpotent, so we can define the exponential as \eqref{formula de exponencial}. For the ind-family of $R$-algebras $\{R[t|\theta^1,\cdots,\theta^n]/\mathfrak{m}^{m}\}_{m\in\mathbb N}$, we get the prounipotent group and its pronilpotent Lie algebra:
	\begin{equation*}
		\begin{split}
			\Aut_{R,+}(R[[ 1|n]]):=&\ \lim_{m\to\infty}\Aut_R(R[t|\theta^1,\cdots,\theta^n]/\mathfrak{m}^{m})\\
			\Der_{R,+}(R[[ 1|n]]):=&\ \lim_{m\to\infty}\Der_R(R[t|\theta^1,\cdots,\theta^n]/\mathfrak{m}^{m})
		\end{split}
	\end{equation*}
and a well defined \emph{exponential}:
	\begin{equation}\label{exponencial 2}
		\exp:\Der_{R,+}(R[[ 1|n]])\to\Aut_{R,+}(R[[ 1|n]]).
	\end{equation}
Observe that this map is surjective, since is surjective for any $m\in\mathbb N$.

Denote by $\Aut_0(R[[ 1|n]])$ the group of automorphisms generated by affine maps on $\{t|\theta^1,\dots,\theta^N\}$.

It was proven \cite[Lemma~6.2.1]{frenkel2004vertex}:

\begin{proposition}\label{descomposicion 0}
	The group $\Aut_R(R[[ 1|n]])$ is a semi-direct product of $\Aut_0(R[[ 1|n]])$ and $\Aut_+(R[[ 1|n]])$.
\end{proposition}

An automorphism $\Phi$ is said to be generated by a vector field $X\in \Der_{R,+}(R[[ 1|n]])$ if $\exp(X)=\Phi$ in \eqref{exponencial 2}.

\subsubsection{The group $\Aut_R^\delta(R[[1|n]])$}

Let $R$ be a super algebra, over $R[[ 1|n]]$ consider the $1|n$-free super module $\Omega^1:=(\Der_{R}(R[[1|n]]))^*$ with generators $\{dt|d\theta^1\cdots d\theta^n\}$, then $\Ber(\Omega^1)$ has a generator $\Delta_0:=[dt|d\theta^1\cdots d\theta^n]$, called \emph{super volume form}. There is a group homomorphism $\Aut_R(R[[1|n]])\to End_R(\Ber(\Omega^1))^*$ given by
	\begin{equation}\label{ber}
		\begin{split}
			\Aut_R(R[[1|n]])\to&\ End_R(\Ber(\Omega^1))^*\\
			\Phi\mapsto&\ \Ber(J\Phi)\text{, where $J\Phi$ is the Jacobian of $\Phi$}.
		\end{split}
	\end{equation}
This homomorphism depends on the basis chosen. Also, $\Phi^*\Delta_0=\Ber(J\Phi)\Delta_0$.

We will denote the kernel of \eqref{ber} by $\Aut_R^\delta(R[[1|n]])$. When there is no confusion we just write $\Aut^\delta[[1|n]]$ and we say that such automorphisms preserve the Berezinian.

For an even vector field $X\in\Der_{R}(R[[ 1|n]])$, differentiating \eqref{ber} we have
	\begin{equation*}
		L_X\Delta_0=\sdiv_{\Delta_0}(X)\Delta_0.
	\end{equation*}
In particular, for a vector field $X\in\Der_{R,+}(R[[ 1|n]])$ we get the relation
	\begin{equation*}
		\Ber(\exp(X))=\exp(\sdiv_{\Delta_0}(X)),
	\end{equation*}
then, $\Phi=\exp(X)$ preserves the Berezinian if and only if $\sdiv(X)=0$. This define the subalgebra
\begin{equation*}
S(1|N)_+:=\Der_{R,+}(R[[ 1|n]])\cap S(1|N).
\end{equation*}
For the group $\Aut^\delta(R[[ 1|n]])$, we will denote by 
\begin{equation*}
\Aut^\delta_0(R[[ 1|n]]):=\Aut_0(R[[ 1|n]])\cap \Aut^\delta(R[[ 1|n]] )
\end{equation*}
and 
\begin{equation*}
\Aut^\delta_+(R[[ 1|n]]):=\Aut_+(R[[ 1|n]])\cap \Aut^\delta(R[[ 1|n]] ),
\end{equation*}
so we get the surjection
\begin{equation*}
S(1|N)_+\stackrel{\exp}{\longrightarrow}\Aut^\delta_+(R[[1|n]]).
\end{equation*}

From Lemma \ref{descomposicion 0} we get: 

\begin{proposition}\label{descomposicion}
		The group $\Aut^\delta_R(R[[ 1|n]])$ is a semi-direct product of $\Aut^\delta_0(R[[ 1|n]])$ and $\Aut^\delta_+(R[[ 1|n]])$.
\end{proposition}

\subsubsection{The group $\Aut_R^\omega(R[[1|n]])$}

Another important group of automorphism is given by the automorphisms preserving the even nondegenerate form
	\begin{equation}\label{forma super simplectica}
		\omega=dz+\theta^1d\theta^1+\cdots+\theta^nd\theta^n.
	\end{equation}

Take an element $\Phi\in\Aut (R[[1|n]])$, we will say that $\Phi\in\Aut_R ^\omega(R[[1|n]])$ if $\Phi$ preserves this form up to multiplication, in other words if $\Phi^*\omega=f\omega$, for some function $f\in R[[1|n]]$. In case there is no confusion, we will write $\Aut ^\omega[[1|n]]$. Observe that $\Aut ^\omega[[1|n]]$ is a super group with the composition as multiplication.

We say that a vector field $X\in K(1|n)$ if $L_X\omega=f\omega$, with $\omega$ as \eqref{forma super simplectica} for some function $f\in R[[1|n]]$ and write $K(1|n)_+:=\Der_{R,+}(R[[ 1|n]])\cap K(1|n)$. We can notice that for a vector field $X\in K(1|n)_+$ we have $\exp(X)\in\Aut ^\omega[[1|n]]$. The group generated by automorphisms $\phi=\exp(X)$, $X\in K(1|n)_+$, is denoted by $\Aut _+^\omega[[1|n]]$, also we have
	\begin{equation*}
		\Aut _+^\omega[[1|n]]=\Aut ^\omega[[1|n]]\cap \Aut _+[[1|n]].
	\end{equation*}

\begin{observation}
	The vector fields generating automorphisms over $R[[ 1|n]]$ preserving \eqref{forma super simplectica}, up to multiplication, are given by:
		\begin{equation*}
			D^f=f\partial_z+\frac{1}{2}(-1)^{j}\sum_{i=1}^{n}(D^if)D^i,
		\end{equation*}
	where $D^i=\theta^i\partial_z+\partial_{\theta^i}$, for any $f\in R[[ 1|n]]_j$. 
	
	For $n\neq 2$, (cf. \cite{manin2014topics}), any change of coordinates that preserves $\omega$, up to multiplication by a function, comes from fields in $K(1|n)$.
	
	When $n=2$, there exists an exterior automorphism given by 
	\begin{equation*}
	(z|\theta^1,\theta^2)\mapsto(z|\theta^2,\theta^1).
	\end{equation*}
	Also, for $n=2$, we have the inclusion $K(1|2)\subset S(1|2)$.
\end{observation}

\section{Super manifolds}
 \label{sec:super-manifolds}

\begin{definition}
	Let $R$ be a super algebra, we define the spectrum $\text{Spec}(R)$ as the set of prime ideals with the Zariski topology. Over $\text{Spec}(R)$ define the sheaf of super algebras $\mathcal O_{R}$ generated by
		\begin{equation*}
			\begin{split}
				\mathcal O_{R}(\text{Spec} R_f)=R_f
			\end{split}
		\end{equation*}
	for any non-nilpotent element $f\in R_0$.
	
	A super scheme is a pair $(M,\mathcal O_M)$ where $M$ is a topological space, $\mathcal O_M$ is a sheaf of super algebras and there exists an open covering of $M$, $\{U_i\}_{i\in I}$, such that $(U_i,\mathcal O_{M}|_{U_i})=(\text{Spec} R_i,\mathcal O_{R_i})$, for some super algebra $R_i$.
	
	We define morphisms of super schemes, sheaves of modules, etc, similar to morphisms of schemes, sheaves of modules, etc.
\end{definition}

\begin{example}
	A scheme $(M,\mathcal O_M)$ defines naturally a super scheme $(M,\mathcal O_M)$ where for any open set $U\subset M$, the super algebra $\mathcal O_M(U)$ is purely even.
\end{example}

\begin{observation}\label{reduccion de haces}
	Let $M$ a topological space and $\mathcal F$ a sheaf of super algebras. The projection given in Observation \ref{reduction} induce a morphism of sheaves given by $\mathcal F(U)\to\mathcal F(U)_{\mathrm{rd}}$, for any open set $U\subset M$. From this projection, for a super scheme $(M,\mathcal O_{M})$ we obtain a super scheme through $\mathcal O_{M}\to\mathcal O_{M,\text{rd}}$ by the pair $M_{\text{rd}}=(M,\mathcal O_{M,\mathrm{rd}})$ and a closed embedding
		\begin{equation*}
			M_{\text{rd}}\hookrightarrow M.
		\end{equation*}
	we will say that $M_{\mathrm{rd}}$ is the \emph{reduced super scheme} of $M$.
\end{observation}
	
\begin{definition}
	A smooth super curve $C$ of dimension $1|n$ is a smooth connected super manifold $(C,\mathcal O_C)$ of dimension $1|n$.
\end{definition}

\begin{observation}\label{variedades split}
	Let $E$ be a locally free sheaf over a smooth scheme $M_0$, we obtain a sheaf of super algebras through Observation \ref{super algebras a traves de modulos}:
		\begin{equation*}
			U\to S_{\mathcal O_{M_0}(U)}[E(U)],
		\end{equation*}
	we will write $S_{\mathcal O_{M_0}}(E)$ and define the super scheme $M=(M_0,S_{\mathcal O_{M_0}}(E))$.
	
	For a super scheme $(M,\mathcal O_{M})$ we will say that it \emph{splits} if is isomorphic to $(M_0,S_{\mathcal O_{M_0}}(E))$. In this case, we have an inclusion
		\begin{equation*}
			\begin{split}
				\mathcal O_{M_0}\to\mathcal O_{M}=S_{\mathcal O_{M_0}}(E)
			\end{split}
		\end{equation*}
	so, there exists a projection
		\begin{equation*}
			M\twoheadrightarrow M_{\text{rd}}.
		\end{equation*}
		
	From now on, we are going to consider just super schemes that are locally split.
	
	For a general family of super schemes $M\to S$ if the base $S$ is not purely even we cannot assure that there exists a projection $M\to M_{\text{rd}}$ such that the following diagram commutes:
		\begin{equation*}\xymatrix{
			M\ar[r]\ar[rd]&\ar[d] M_{\text{rd}}\\
			&S}
		\end{equation*}
	When such projection exists we are going to say that $M$ is \emph{projected}.
\end{observation}

\begin{definition}
	A \emph{super manifold} is a super scheme $(M,\mathcal O_M)$, such that the sheaf of $\mathcal O_M$-modules given by $\text{Der}(\mathcal O_M):=\cT_M$ is a locally free sheaf of $\mathcal O_M$-modules. For a super manifold $M$, we say that it has dimension $m|n$ if $\cT_M$ has rank $m|n$. A \emph{super curve} is a connected $1|n$-super manifold.
	
	For a family $M\to S$ we say that is a \emph{family of super manifolds} if $\text{Der}_{\mathcal O_S}(\mathcal O_M):=\cT_{M/S}$ is locally free, and that have relative dimension $m|n$ if $\cT_{M/S}$ has rank $m|n$. Similarly, we define a \emph{family of curves over a super scheme} $S$.
\end{definition}

\begin{observation}
	For a closed point, $p\in M$, there exists an open set $U$ such that $\mathcal O_M(U)=S_{\mathcal O_{M,\text{rd}}}(E)$, for some $O_{M,\text{rd}}(U)$-free module $E$.
\end{observation}

\begin{example}
	Let $M$ an $m|n$ a super manifold, the tangent and cotangent bundle are rank $m|n$ locally free $\cO_M$-modules. Observe, that if we have local coordinates over an open set $U$, given by $(z_1,\dots,z_m|\theta^1,\dots,\theta^N)$, then the tangent space is locally trivialized, by $\langle\partial_{z_1},\dots,\partial_{z_m}|\partial_{\theta^1},\dots,\partial_{\theta^N}\rangle$, and the cotangent space is locally trivialized, by $\langle d{z_1},\dots,d{z_m}|$ $d{\theta^1},\dots,d{\theta^N}\rangle$.
\end{example}

\begin{definition}
	For a super manifold $M$ with a rank $m|n$ locally free sheaf $E$ we can define the \emph{Berezinean sheaf} $\Ber(E)$ as follows: for an open set $U\subset M$ such that $E(U)$ is free $\mathcal O_M(U)$-module we define $\Ber(E)(U)$ as Observation \ref{modulo bereziniano}.
	
	For two open sets $U,V$ and change of coordinates $\Phi$, then the $\Ber(J\Phi)$ gives us the cocycle of $\Ber(E)$. The sheaf $\Ber(E)$ has rank $1|0$ if $n$ is even and rank $0|1$ if $n$ is odd.
	
	Set $\Ber_M:=\Ber(\Omega_M)$, and for a family $M\to S$, consider $\Ber_{M/S}:=\Ber(\Omega_{M/S})$.
\end{definition}

\begin{example}
	Let $M=(M_0,S_{\mathcal O_{M_0}}(E))$ be a split super manifold, then we can take the local coordinates over an open set $U\subset M_0$, $\{z_1,\dots,z_m|\theta^1,\dots,\theta^n\}$, in this coordinates the cotangent space is locally trivialized, over $U$, by $\langle d{z_1},\dots,d{z_m}|d{\theta^1},\dots,d{\theta^n}\rangle$, so $\Ber_M(U)= \mathcal O_M(U) [d{z_1},\dots,d{z_m}|d{\theta^1},\dots,d{\theta^n}]$. For another coordinates $\{w_1,\dots,$ $w_m|\rho^1,\dots,\rho^n\}$ with 
		\begin{equation*}
			\begin{split}
				w_i=&\ \phi_i(z_1,\dots,z_m)\\
				\rho^j=&\ \theta^1a_{j1}(z_1,\dots,z_m)+\dots+\theta^Na_{jN}(z_1,\dots,z_m),\ j=1,\dots,N.
			\end{split}
		\end{equation*}
		
	The change of coordinates of the cotangent bundle is given by
		\begin{equation*}
			\left[\begin{matrix}
			\partial_{z_k}\phi_l&B\\
			0&a_{ij}
			\end{matrix}\right],
		\end{equation*}
	where $B=(\partial_{z_i}\rho^j)$, then the change of coordinates of the cotangent bundle is given by
		\begin{equation*}
			\det (\partial_{z_k}\phi_l)\det(a_{ij})^{-1}.
		\end{equation*}
	Using the closed embedding, $j:M_0\to M$, we get the isomorphism
		\begin{equation*}
			j^*\Ber_M\simeq \Omega^m_{M_0}\otimes \det E^*.
		\end{equation*}
	Observe that in this case $\Ber_M$ is a trivial bundle when $\Omega^m_{M_0}\simeq \det E$ as line bundles over $M_0$.
\end{example}

Let $C$ be a super curve, for a section $\Delta\in H^0(C,\Ber_C)$ and coordinate patch $\{(U_i,\Phi_i)\}$, then there exists a family of functions $f_i\in H^0(U_i,\mathcal O_{C})$ such that
	\begin{equation}\label{trivializacion}
		\Delta|_{U_i}=f_i[dz_i|d\theta_i^1\cdots d\theta_i^N].
	\end{equation}


\begin{observation}
	Let us consider a $1|n$-super curve $C$, a coordinate patch $U\subset C$ with a trivialization $\Phi$, and a nonvanishing section $\Delta\in H^0(U,\Ber_C)$ for some $f\in H^0(U,\mathcal O_{C}^*)$ as \eqref{trivializacion}. Taking an even function $F(z|\theta^1,\cdots ,\theta^n)$, and shrinking $U$ if is necessary, with
		\begin{equation*}
			\partial_zF(z|\theta^1,\cdots,\theta^n)=f(z|\theta^1,\cdots,\theta^n),
		\end{equation*}
	then the system of coordinates $\Psi=(w|\rho^1,\cdots,\rho^n)$, given by 
		\begin{equation*}
			\begin{split}
				w  = &\ F(z|\theta^1,\cdots ,\theta^N),\\
				\rho^i = & \ \theta^i,\ i=1\dots,n;
			\end{split}
		\end{equation*}
	verifies $\Delta=[dw|d\rho^1\cdots d\rho^n]$. We say that such coordinate system $\Psi$ is \emph{compatible} with the section $\Delta$.
	
	Finally, for a nonvanishing section $\Delta\in H^0(C,\Ber_C)$, then there exists a coordinate system of $C$, $\{(U_i,\Phi_i)\}_i$ such that
		\begin{equation*}
			\Delta|_{U_i}=[dz_i|d\theta_i^1\cdots d\theta_i^N].
		\end{equation*}
		
	For any pair of coordinates $\Phi$, $\Psi$ defined over the same open set $U$ both compatible with $\Delta|_{U}$, then the change of coordinates preserves the Berezinian.
	
	For a fixed curve $C$ and a nonvanishing section $\Delta\in H^0(C,\Ber_C)$ we will only consider coordinates compatible with $\Delta$.
\end{observation}
	
\subsection{$S(2)$-super curves}

Fix a base super scheme $S$, we will consider curves and bundles relative to $S$.

\begin{definition}
	An \emph{$S(1|2)$-super curve} is a pair $(C,\Delta)$, where $C\to S$ is a super curve and a nonvanishing section $\Delta\in H^0(C,\Ber_{C/S})$. 
\end{definition}

From Proposition \ref{descomposicion}, for any change of coordinates $\Phi\in \Aut ^\delta(R[[1|2]])$ there exists a divergence free field $X\in S(1|2)_+$ and $T\in\Aut ^\delta_0(R[[1|2]])$ such that $\Phi(z|\theta^1,\theta^2)=\exp(X) (T(z|\theta^1,\theta^2))$. Our interest is to study such automorphisms where $X\in S(2)_+:=S(2)\cap \Der_{R,+}(R[[1|2]])$, when this happens we write $\Phi\in\Aut ^{\Delta}(R[[1|2]])$ and observe that $\Aut ^{\Delta}(R[[1|2]])$ is a subgroup of $\Aut ^\delta(R[[1|2]])$ where
	\begin{equation*}
		\Aut ^\delta_0(R[[1|2]]) < \Aut ^\Delta(R[[1|2]]) < \Aut ^\delta(R[[1|2]]).
	\end{equation*}
When there is no confusion, we write $\Aut ^\Delta[[1|2]]=\Aut ^\Delta(R[[1|2]])$.
	
\begin{definition}
	An \emph{$S(2)$-super curve} is an $S(1|2)$-super curve $(C,\Delta)$ such that there exists a system of coordinates $\{U_i,\Phi_i\}$ compatible with $\Delta$ and change of coordinates $\Phi_{ij}=\Phi_i\circ\Phi^{-1}_j\in \Aut ^\Delta[[1|2]]$. We say that a family of curves $C\to S$ is a family of $S(2)$-super curves, if for any $y\in S$, $C_y$ is an $S(2)$-super curve.
\end{definition}	

Over a curve $C\to S$ with a nonvanishing section $\Delta\in H^0(C,\Ber_{C/S})$, over an open set we can define the space of vector fields:
	\begin{equation*}
		S(1|2)(U):=\{X\in\cT_C(U):\mathrm{sdiv}_{\Delta}X=0\},
	\end{equation*}
and
	\begin{equation*}\label{fibrado S(2)}
		S(2)(U):=[S(1|2)(U),S(1|2)(U)].
	\end{equation*}
Observe that such spaces do not define a sheaf of $\mathcal O_C$-modules. They define, however, a sheaf of $\pi^*\mathcal O_S$-modules.

Finally, we get a $\pi^*\mathcal O_S$-module $K$ defined over an open set $U\subset C_0$:
	\begin{equation*}
		K(U):=\frac{S(1|2)(U)} {S(2)(U)}.
	\end{equation*}
On the other hand, we have the isomorphisms
	\begin{equation*}
		\begin{split}
			\exp:S(1|2)_+\to\Aut ^\delta_+(R[[ 1|n]]),\\
			\exp:S(2)_+\to\Aut ^\Delta_+(R[[ 1|n]]),
		\end{split}
	\end{equation*}
and $\Aut ^\delta_0(R[[ 1|n]])=\Aut ^\Delta_0(R[[ 1|n]])$, so, we have the isomorphism:
	\begin{equation}\label{isomorfismo}
			\exp:\frac{S(1|2)}{S(2)}\to\frac{\Aut ^\delta(R[[ 1|n]])}{\Aut ^\Delta(R[[ 1|n]])}\simeq \mathbb G_a.
	\end{equation}

\begin{observation}
	For an $S(1|2)$-super curve $(C,\Delta)$ we can construct the bundle of coordinates preserving the Berezinian, $\Aut ^\delta_C$, considered as the set of pairs $(Z,\Phi)$ for $Z$ a $S$-point in $C$ and $\Phi$ a local system of coordinates compatible with the section $\Delta$. This bundle is an $\Aut ^\delta(R[[ 1|2]])$-bundle, and observe that the quotient:
		\begin{equation}\label{fibrado}
			\Aut^\Delta[[1|2]]\backslash\Aut^\delta_C\to C
		\end{equation}
	is an $\mathbb G_a$-bundle. From \eqref{isomorfismo}, we get that $\Aut^\Delta[[1|2]]\backslash\Aut^\delta_C$ is isomorphic to $K$.
\end{observation}

Now, we can reformulate the definition of $S(2)$-super curves:

\begin{proposition}\label{fibrado trivial es S2}
	An $S(2)$-super curve is an $S(1|2)$-super curve $(C,\Delta)$ such that the $\mathbb G_a$-bundle \eqref{fibrado} is trivial.
\label{prop:trivialga}
\end{proposition}

\begin{proof}
	Observe that the bundle \eqref{fibrado} is trivial if and only if it has a section. In this case, a section is a covering with trivializations $\{(U_i,\Phi_i)\}_{i}$ such that the change of coordinates $\Phi_{ij}:=\Phi_{j}\circ\Phi_{i}^{-1}\in \Aut ^\Delta(R[[ 1|2]])$.
	
	Finally, the bundle is trivial if and only if there exists a covering for $C$ with trivializations $\{(U_i,\Phi_i)\}_{i}$ compatible with $\Delta$ such that the change of coordinates $\Phi_{ij}\in \Aut ^\Delta[[ 1|2]]$, that is, $(C,\Delta)$ is an $S(2)$-super curve.
\end{proof}

\begin{observation}
	For a $1|2$-super curve over an even base $S$, $C\to S$, it is shown in \cite{Noja:2018edj} that there exists a similar class that measures if the curve is split. Suppose that $C$ is a $1|2$-super curve over a point with reduced curve $C_0$, then we have a sequence
		\begin{equation*}
			0\to \mathcal J \to \mathcal O_C\to \mathcal O_{C_0}\to 0,
		\end{equation*}	
	for which we have the inclusion $C_0\stackrel{j}{\to}C$. Now considering the curve $\overline C=(C_0,\mathcal O_{C,0})$ and the inclusion $C_0\stackrel{j}{\to}\overline C$ we get the sequence of sheaves of algebras over $\overline C$:
		\begin{equation}\label{secuencia exacta}
			0\to \det\mathcal F \to \mathcal O_{C,0}\to \mathcal O_{C_0}\to 0,
		\end{equation}
	for the rank 2 bundle $ \mathcal F= \mathcal J/ \mathcal J^2$, in $\det\mathcal F$ we consider the zero multiplication. Taking local splits $\pi_i:\mathcal O_{C_0}(U_i)\to \mathcal O_{C,0}(U_i)$ in \eqref{secuencia exacta} we can define
		\begin{equation}\label{clase}
			\omega_{ij}=\pi_i|_{U_i\cap U_j}-\pi_j|_{U_i\cap U_j}
		\end{equation}	
	observe that 
		\begin{equation*}
			\begin{split}
				\omega_{ij}(fg)=&\pi_i(fg)-\pi_j(fg)\\
				=&\pi_i(f)\pi_i(g)-\pi_j(f)\pi_i(g)\\
				=&\pi_i(f)\omega_{ij}(g)+\omega_{ij}(f)\pi_i(g)
			\end{split}
		\end{equation*}
	so $\omega_{ij}\in\cT_{C_0}\otimes\det\mathcal F(U_{ij})$. Additionally, $\omega_{ij}$ verifies the cocycle condition:
		\begin{multline*}
				\omega_{ij}+\omega_{jk}+\omega_{ki}=(\pi_i|_{U_i\cap U_j\cap U_k}-\pi_j|_{U_i\cap U_j\cap U_k})+(\pi_j|_{U_i\cap U_j\cap U_k}-\pi_k|_{U_i\cap U_j\cap U_k})\\
				+(\pi_k|_{U_i\cap U_j\cap U_k}-\pi_i|_{U_i\cap U_j\cap U_k}) = 0
		\end{multline*}
	so $\{\omega_{ij}\}\in H^1(C_0,T_{C_0}\otimes\det\mathcal F(U_{ij}))$.
	
	For an $S(1|2)$-super curve we have $\det\mathcal F=\Omega_{C_0}$, then $\{\omega_{ij}\}\in H^1(C_0,T_{C_0}\otimes\Omega_{C_0})=H^1(C_0,\mathcal O_{C_0})$. Let $\{\Phi_i=(z_i|\theta^1_i,\theta^2_i)\}$ local coordinates over $C$. For the change of coordinates $\phi_{ij}$ we get
		\begin{equation*}
			\begin{split}
				z_j=&\ F_{ij}(z_i)+G_{ij}(z_i)\theta^1_i\theta^2_i\\
				\theta^1_j=&\ \theta^1_ia_{11}(z_i)+\theta^2_ia_{12}(z_i)\\
				\theta^2_j=&\ \theta^1_ia_{21}(z_i)+\theta^2_ia_{22}(z_i).
			\end{split}
		\end{equation*}
	Since $C$ has a trivial Berezinian, then $G_{ij}=\lambda_{ij}\partial_{z_i}F_{ij}$, for $\lambda_{ij}$ a constant, then we have
		\begin{equation*}
			\begin{split}
				z_j=&\ F_{ij}(z_i)+\lambda_{ij}\partial_{z_i}F_{ij}(z_i)\theta^1_i\theta^2_i=F_{ij}(z_i+\lambda_{ij}\theta^1_i\theta^2_i)\\
				\theta^1_j=&\ \theta^1_ia_{11}(z_i)+\theta^2_ia_{12}(z_i)=g^1_{ij}(z_i|\theta^1_i,\theta^2_i)\\
				\theta^2_j=&\ \theta^1_ia_{21}(z_i)+\theta^2_ia_{22}(z_i)=g^2_{ij}(z_i|\theta^1_i,\theta^2_i).
			\end{split}
		\end{equation*}
 
	From the generators given in Observation \ref{generators}, we get that $\widetilde{\Phi}_{ij}=(F_{ij}|g^1_{ij},g^2_{ij})\in\Aut^\Delta[[1|2]]$ and $z_i+\lambda_{ij}\theta^1\theta^2_i=\exp(\lambda_{ij}\theta_i^1\theta^2_i\partial_{z_i})(z_i)$, then $\Phi_{ij}=\widetilde{\Phi}_{ij}\circ\exp(\lambda_{ij}\theta_i^1\theta^2_i\partial_{z_i})$, that is $\exp(\lambda_{ij}\theta_i^1\theta^2_i\partial_{z_i})$ gives the cocycle in \eqref{fibrado}.
	
	The relation of \eqref{clase} with the class \eqref{fibrado} is the following:
		\begin{equation*}
			\omega_{ij}(f)=\lambda_{ij}\theta_i^1\theta_i^2\partial_{z_i}f
		\end{equation*}
	for local coordinates $(z_i|\theta_i^1,\theta_i^2)$ over $U_i$.
	
	 In \cite{Noja:2018edj} it is proved that $C\to S$ is projected if and only if $\{\omega_{ij}\}\in H^1(C,\cO_{C/S})$ vanishes. Then we obtain
	 
	 \begin{theorem}\label{split}
	 	Every $S(2)$-super curve over a purely even base $S$ is split.
	 \end{theorem}
	
	In order to get a geometric interpretation of this,consider a $1|2$-super curve $C$, the inclusion $C_0\stackrel{j}{\to}\overline C$ and the space of differentials over $\Omega_{\overline C}$, we obtain that $j^*\Omega_{\overline C}$ is a rank 2 bundle over $C_0$ with a projection $j^*\Omega_{\overline C}\to\Omega_{ C_0}\to 0$. Actually, we get the sequence of $\mathcal O_{C_0}$-modules:
		\begin{equation}\label{extension}
			0\to \det \mathcal F \to j^*\Omega_{\overline C}\to\Omega_{ C_0}\to 0.
		\end{equation}
	As an extension of $\mathcal O_{C_0}$-modules, \eqref{extension}, is defined by an element of
		\begin{equation*}
			\text{Ext}^1(\Omega_{ C_0},\det \mathcal F)=H^1(\Omega_{ C_0}^*\otimes\det \mathcal F )=H^1(\cT_{ C_0}\otimes\det \mathcal F ),
		\end{equation*}
	and such element is the class $\{\omega_{ij}\}$ defined above.
	
	Now, if $ \det \mathcal F=\Omega_{C_0}$, then we have the sequence
		\begin{equation}\label{extension 2}
			0\to \Omega_{ C_0}\to j^*\Omega_{\overline C}\to\Omega_{ C_0}\to 0.
		\end{equation}
	
	This occurs when $C$ is a $S(1|2)$ curve, but to distinguish the element $K\in H^1(C_0,\pi^*\cO_S)\subseteq H^1(C_0,\mathcal O_{C_0})$, where $\pi:C\to S$, we have to notice that the sequence \eqref{extension 2} fits in the following diagram:
		\begin{gather}
			\begin{aligned}
				\xymatrix{
					&0&0&0&\\
					0 \ar[r] &\Omega_{C_0} \ar[r] \ar[u]  & L\otimes\Omega_{C_0} \ar[r]  \ar[u] & \Omega_{C_0}\ar[r] \ar[u] & 0\\
					0 \ar[r] &\mathcal O_{C_0} \ar[r]^{\alpha\hspace{0.35cm}} \ar[u]^d & L\otimes\mathcal O_{C_0} \ar[r] \ar[u]^{1\otimes d} & \mathcal O_{ C_0}\ar[r] \ar[u]^d & 0\\
					0 \ar[r] &\pi^*\cO_S \ar[r] \ar[u]  & L \ar[r] \ar[u] &\pi^*\cO_S\ar[r] \ar[u] & 0\\
					&0\ar[u] &0\ar[u] &0\ar[u] & }
			\end{aligned}
			\label{Fig1}
		\end{gather}
	where $L\otimes \Omega_{C_0}=j^*\Omega_{\overline C}$ and $L$ as an extension of $\pi^*\cO_S$-modules represents the class $\Gamma_C\in H^1(C_0,\pi^*\cO_S)$.
	
	Finally, in \cite{Nojathesis} and \cite{Donagi:2014hza} it is proved that each $1|2$ super curve, over a point, is defined by the data of $(C_0,\mathcal F,\{\omega_{ij}\})$, where $\{\omega_{ij}\}$ represents an extension of $\mathcal O_{C_0}$-modules:
		\begin{equation*}
			0\to \det \mathcal F \to \mathcal L\to\Omega_{ C_0}\to 0.
		\end{equation*}
	For an $S(1|2)$ we need $(C_0,\mathcal F,\Gamma_C)$, where $\Gamma_c$ represents an extension of $\pi^*\cO_S$-modules
		\begin{equation*}
			0 \to\pi^*\cO_S \to L \to\pi^*\cO_S\to 0
		\end{equation*}
	which gives rise to the diagram \eqref{Fig1}.
\end{observation}
	
\subsection{$SUSY$-super curves}

\begin{definition}
	A $1|n$-super curve $C$ with a covering and coordinate systems $\{\phi_i\}_i$ such that any change of coordinates verifies $\phi_{ij}=\Phi_i\circ\Phi^{-1}_j\in\Aut ^\omega[[1|n]]$ we will say that $C$ is a \emph{$SUSY_n$-super curve}. 
	
	Observe that the local form \eqref{forma super simplectica} in coordinates $\phi_i=(z_i|\theta^1_i,\dots,\theta^n_i)$
		\begin{equation*}
			\omega_i=dz_i+\theta^1_id\theta^1_i+\cdots+\theta^n_id\theta^n_i 
		\end{equation*}
	is well defined, up to multiplication by a function over $C$. 
	
	Also, we can define an $SUSY_n$-super structure over the $1|n$-super curve $C$ as a locally free subsheaf $E \subset\cT_C$ of rank $0|n$, for which the Frobenius form 
		\begin{equation*}
			E\otimes E \to\cT_C/E	
		\end{equation*}
	is nondegenerate and split, i.e., it locally has an isotropic direct subsheaf of maximal possible rank $k$ for $n = 2k$ or $2k + 1$ (cf. \citealp{manin2014topics}). 
\end{definition}

\begin{observation}
	For a $1|2$-super curve $C$, if each change of coordinates comes from a $K(1|2)$ field we say that such curve is orientable. Oriented curves are also characterized by the fact that the distribution $E$ decomposes as
		\begin{equation*}
			E=\mathscr L\oplus \mathscr L^*\otimes\Omega_C,
		\end{equation*}
	for some line bundle $\mathscr{L}$ over $C_0$.
	
	Observe that in the non-oriented case such bundle $\mathscr{L}$ only exists locally.
	
	Since $K(1|2)\subseteq S(2)$, then any change of coordinates that comes from a field on $K(1|2)$ preserves the Berezinian and actually any orientable $SUSY_2$-super curve is an $S(2)$-super curve.	
\end{observation}
	
\begin{observation}
	Suppose that $N=2n$, over $R[[1|2n]]$ consider the change of coordinates
		\begin{equation*}
			\begin{split}
				w=&z+i\left(\theta^{1}\theta^{2}+\cdots+\theta^{2n-1}\theta^{2n}\right),\\
				\rho^j=&-i(-\theta^{2j-1}+i\theta^{2j}),\ j=1\dots,n,\\
				\eta^j=&-i(\theta^{2j-1}+i\theta^{2j}),\ j=1\dots,n,
			\end{split}
		\end{equation*}
	we obtain that $\omega=dz+\theta^1d\theta^1+\cdots+\theta^nd\theta^n$ changes as
		\begin{equation}\label{otra estructura N=2n}
			\widetilde{\omega}=dz+\rho^1d\eta^1+\cdots+\rho^nd\eta^n
		\end{equation}
	so, the group of automorphisms of $R[[1|2n]]$ that preserve $\omega$ up to multiplication by a function coincides with the group of automorphisms of 
		\begin{equation*}
			R[[ w|\rho^1,\cdots,\rho^n,\eta^1,\cdots,\eta^n]]
		\end{equation*}
	that preserve \eqref{otra estructura N=2n} up to multiplication by a function.
	
	Observe that for a $1|2n$-super curve $C$ for which \eqref{otra estructura N=2n} is well defined, up to multiplication by a function, then $C$ is a $SUSY$-super curve, the converse is not necessarily true.
\end{observation}

Consider the inclusion $R[[1|n]]\hookrightarrow R[[1|2n]]$ given by the identification
	\begin{equation*}
		R[[1|2n]]=R[[1|n]] [\rho^1,\dots,\rho^n].
	\end{equation*}
Let $\Phi=(F|\phi^1,\dots,\phi^n)\in \Aut [[ 1|n]]$ and consider the super function $\widetilde{\Phi}=(F|\phi^1,\dots,\phi^n,\eta^1,\dots,\eta^n)$. The pullback of the form \eqref{otra estructura N=2n} is given by
	\begin{equation*}
		\begin{split}
			\widetilde{\Phi}^*(dz+\rho^1d\theta^1+\dots+\rho^nd\theta^n) = & \partial_{z}Fdz-\partial_{\theta^1}Fd\theta^1-\dots -\partial_{\theta^n}Fd\theta^n \\
			&+\eta^1(\partial_{z}\phi^1dz+\partial_{\theta^1}\phi^1d\theta^1+\dots +\partial_{\theta^n}\phi^1d\theta^n)+\cdots\\
			&\hspace{0.9cm} + \eta^n(\partial_{z}\phi^ndz+\partial_{\theta^1}\phi^nd\theta^1+\dots +\partial_{\theta^n}\phi^nd\theta^n)\\
			= & (\partial_{z}F+ \eta^1\partial_{z}\phi^1+\cdots+\eta^n\partial_{z}\phi^n)dz\\
			&\hspace{0.4cm}+(-\partial_{\theta^1}F+\eta^1\partial_{z}\phi^1+\cdots+\eta^n\partial_{z}\phi^n)d\theta^1+\cdots\\
			&\hspace{1.4cm}+(-\partial_{\theta^n}F+\eta^1\partial_{z}\phi^1+\cdots+\eta^n\partial_{z}\phi^n)d\theta^n.
		\end{split}
	\end{equation*}

Then, defining the differential operators $D_j=\rho^j\partial_z+\partial_{\theta^j}$, the function $\widetilde\Phi\in\Aut^\omega[[1|2n]]$ if and only if 
	\begin{equation}\label{coordenadas n a 2n} 
		D_iF=\eta^1D_i\phi^1+\cdots+\eta^nD_i\phi^n,\ i=1,\dots,n. 
	\end{equation}
Since the matrix $(D_i\phi^j)_{ij}$ is invertible, we get:
	\begin{equation}\label{coordenadas susy 2k}
		\left[\begin{matrix} D_1\phi^1&\cdots & D_1\phi^n\\\vdots& &\vdots\\D_n\phi^1&\cdots & D_n\phi^n\end{matrix}\right]^{-1} \left[\begin{matrix} D_1 F\\\vdots\\D_n F\end{matrix}\right] = \left[\begin{matrix}\eta^1\\\vdots\\\eta^n\end{matrix}\right].
	\end{equation}	
With these coordinates, $\widetilde\Phi=(F|\phi^1,\dots,\phi^n,\eta^1,\dots,\eta^n)$ makes the following diagram commutes
	\begin{equation}\label{maestro}
		\begin{split}
			\xymatrixcolsep{4pc} \xymatrix{
				R[[1|n]] \ar[d] \ar[r]^{\Phi} & R[[1|n]] \ar[d]\\
				R[[1|2n]] \ar[r]^{\widetilde{\Phi}}& R[[1|2n]]}
		\end{split}
	\end{equation}
and we obtain the inclusion of groups:
	\begin{equation}\label{inclusion}
		\Aut[[ 1|n]]\stackrel{j}{\hookrightarrow}\Aut ^\omega[[ 1|2n]].
	\end{equation}
Also, we obtain an inclusion of Lie algebras given by
	\begin{equation*}
		\Der_{R}(R[[ 1|n]])\stackrel{\widehat{j}}{\hookrightarrow} K(1|2n),
	\end{equation*}
in such way that for $X\in \Der_{R,+}(R[[ 1|n]])$ we obtain $j(\exp(X))=\exp(\widehat{j}(X))$.
	
\begin{proposition}\label{coordenadas n a Susy 2n}
	Any $1|n$-super curve has an $SUSY_{2n}$-super curve associated.
\end{proposition}

\begin{proof} 
	Given a super curve $C$, consider the $\Aut[[1|n]]$-principal bundle $\Aut_C$, then we can construct an $\Aut^\omega[[1|2n]]$-principal bundle given by
		\begin{equation*}
			\Aut^\omega[[1|2n]]\times_{\Aut[[1|n]]}\Aut_C
		\end{equation*}
	for the inclusion $\Aut[[1|n]]\subset\Aut^\omega[[1|2n]]$ given in \eqref{inclusion}.
	
	The structure of $1|n$-super curve give us  family of local sections of the bundle $\Aut_C\to C$, $\{\phi_i\}_{i\in I}$; then the family of local sections $\{\Phi_i=(1,\phi_i)\}_{i\in I}$ gives us a family of $K(1|2n)$-super curves.
\end{proof}

The curve obtained in Proposition \ref{coordenadas n a Susy 2n} is going to be denoted by $\widetilde C$. If we have an atlas $\{\phi_i=(z|\theta^1,\theta^2)\}_{i\in I}$ over $C$ with cocycles $\phi_{ij}=\Phi_i\circ\Phi^{-1}_j$, we construct the atlas over $\widetilde C$ given by $\{\widetilde\phi_i:=(z|\theta^1,\theta^2,\rho^1,\rho^2) \}$, with cocycles $\widetilde\phi_i\widetilde\phi^{-1}_j=j(\phi_{ij})$, with $j$ given in \eqref{inclusion}.
	
\begin{observation}In the previous construction, we get a projection locally given by
	\begin{equation*}
		(z|\theta^1,\dots,\theta^n,\rho^1,\dots,\rho^n) \mapsto (z|\theta^1,\dots,\theta^n).
	\end{equation*}

By construction, the projection $\widetilde{C}\stackrel{\pi}{\longrightarrow} C$ is well defined. For any point $p\in C$ the fiber has dimension $0|n$.
\end{observation}
	
\begin{observation}
	There exists a geometric description of this fact given by \cite{dolgikh1990}. The space $\widetilde{C}$ is described by the space of $0|n$-subspaces of $T_pC$ for any $p\in C$. The $SUSY$-super structure for a point $(p,E)$ is given by the distribution $\widetilde{E}\subset\cT_{(p,E)}\widetilde{C}$ defined by the local form $dz+\rho^1d\theta^1+\dots+\rho^nd\theta^n$ for $p=(z|\theta^1,\dots,\theta^n)$. Locally, $\widetilde{E}$ is generated by $\{\partial_{\rho^1},\dots,\partial_{\rho^n},\rho^1\partial_z+\partial_{\theta^1},\dots,\rho^n\partial_z+\partial_{\theta^n}\}$.
	
	In this context, the projection $\widetilde{C}\stackrel{\pi}{\longrightarrow} C$ is given by $(p,E)\mapsto p$ and the distribution $\widetilde{E}$ is given by $d\pi^{-1}(E)$, for $T_{(p,E)}\widetilde{C}\stackrel{d\pi}{\longrightarrow}T_pC$.
\end{observation}
	
For  $SUSY_n$-super curve $C\to S$ the operators $D^i=\theta^i\partial_z+\partial_i$ define a $0|n$-distribution over $T_{C/S}$ with $	D_\alpha^i = (D_\alpha^i \phi_{\alpha,\beta}^1) D_\beta^i+\cdots+(D_\alpha^i \phi_{\alpha,\beta}^n) D_\beta^i,\ i=1,\dots,n$. For the change of coordinates $(F|\phi^1,\dots,\phi^n)$ the operators $D^i$, for $i=1,\dots,n$, verify
	\begin{equation}\label{ecuacion de N=k SUSY curvas}
		D^i F = \phi^1D^i \phi^1+\cdots+\phi^nD^i \phi^n,\ i=1,\dots,n.
	\end{equation}
	
\begin{observation}	
	The $SUSY_{2n}$-super curve $\widetilde{C}$ associated to the $1|n$-super curve $C$ has a rank $0|n$ bundle locally generated by the local fields $D^i=\partial_{\rho^i}$, $i=1,\dots,n$. Reciprocally, suppose that $C$ is a $SUSY_{2n}$-super curve with the local coordinates $(w|\theta^1,\cdots,\theta^{2n})$ with a change of coordinates $(F|\phi^1,\cdots,\phi^n)$, introducing the new variables	
		\begin{equation*}
			\begin{split}
				w = &\ z+i\left(\theta^{1}\theta^{2}+\cdots+\theta^{2n-1}\theta^{2n}\right)\\
				\zeta^j  = &-i(-\theta^{2j-1}+i\theta^{2j}),\ j=1,\dots,n.\\
				\rho^j  = &  -i(\theta^{2j-1}+i\theta^{2j}),\ j=1,\dots,n.
			\end{split}
		\end{equation*}
	The change of coordinates 
		\begin{equation*}
			\begin{split}
				G=&F +i\left(\phi^{1}\phi^{2}+\cdots+\phi^{2n-1}\phi^{2n}\right)\\
				\psi^j=&-i(-\phi^{2j-1}+i\phi^{2j}),\ j=1,\dots,n,\\
				\eta^j=&-i(\phi^{2j-1}+i\phi^{2j}),\ j=1,\dots,n,
			\end{split}
		\end{equation*}
	and considering the operators $D^j_\pm=\frac{1}{2}\left\{D^{2j}\pm iD^{2j-1} \right\}$, the equation \eqref{ecuacion de N=k SUSY curvas} reads:
		\begin{equation}\label{ecuacion de otras coordenadas de N=k SUSY curvas}
			D^j_\pm G  = \eta^1D^j_\pm \psi^1+\cdots+\eta^nD^j_\pm \psi^n.
		\end{equation}
	This induces the rank $0|2n$-distribution
		\begin{equation}\label{fibrado y subfibrado}
			D^j_{\alpha,\pm}  = \sum_{k=1}^n \left((D^j_{\alpha,\pm} \psi^k )D^j_{\beta,-} +(D^j_{\alpha,\pm} \eta^k) D^j_{\beta,+} \right)
		\end{equation}
	and if the following equations hold
		\begin{equation}\label{equacion de orientacion}
			D^j_+\psi^k = 0,\qquad j,k=1,\dots,n;
		\end{equation}
	then we can define the $1|n$-super curve by the coordinates $(w|\zeta^1,\dots,\zeta^n)$. The conditions described in equation \eqref{equacion de orientacion} are equivalent to having the $0|n$-distribution locally defined by $D^j_+$, $j=1,\dots,n$.
	
	Let $C$ be a $SUSY_4$-super curve. Taking the $\Aut^\omega[[1|4]]$-principal bundle $\Aut_C\to C$, we obtain that it comes from a $1|2$-super curve if and only if the bundle $\Aut[[1|2]]\backslash\Aut_C$ has a global section.
\end{observation}
	
\begin{example}
	The $1|1$-super curves induce what are called oriented $SUSY_2$-super curves. For a general $SUSY_2$-super curve and change of coordinates $(F|\psi,\eta)$ equations \eqref{fibrado y subfibrado} are:
		\begin{equation*}
			\begin{split}
				D_{\alpha,+}= & D_{\alpha,+} \psi D_{\beta,-} +D_{\alpha,+} \eta D_{\beta,+}, \\
				D_{\alpha,-} = & D_{\alpha,-} \psi D_{\beta,-} +D_{\alpha,-} \eta D_{\beta,+},
			\end{split}
		\end{equation*}
	and	since the matrix
		\begin{equation*}
			\left[\begin{matrix}D_{\alpha,+} \psi & D_{\alpha,+} \eta \\
				D_{\alpha,-} \psi & D_{\alpha,-} \eta \end{matrix}\right]
		\end{equation*}
	is invertible, then $D_{\alpha,-} \psi D_{\alpha,+} \eta\neq 0 $. If \eqref{equacion de orientacion} holds, differentiating equation \eqref{ecuacion de otras coordenadas de N=k SUSY curvas}, we get 
		\begin{equation*}
			\begin{split}
				D_{\alpha,-}D_{\alpha,-}F =&\ D_{\alpha,-}\left(\eta  D_{\alpha,-} \psi\right),\\
				=&\ D_{\alpha,-}\eta  D_{\alpha,-} \psi-\eta D_{\alpha,-} D_{\alpha,-} \psi;
			\end{split}
		\end{equation*}
	since $D_{\alpha,-} D_{\alpha,-} =0$, we obtain	that $ D_{\alpha,-} \eta D_{\alpha,-} \psi=0$, then $D_{\alpha,-} \eta$, so there exists another bundle defined by $D_-$ that for a change of coordinates $(F|\psi,\eta)$ we get
		\begin{equation*}
			D_{\alpha,-} = D_{\alpha,-} \psi D_{\beta,-}.
		\end{equation*}
	This line bundle induces another curve $\widehat{C}$ defined by the coordinates:
		\begin{equation*}
			\begin{split}
				\widehat z = & z-\theta \rho,\\
				\widehat \rho  = & \rho,
			\end{split}
		\end{equation*}
	this curve is called the ``dual" curve associated to $C$. The situation is described in \cite{manin2014topics}.
\end{example}
	
\begin{example}
	Over the affine plane $\mathbb C^{1|4}$ consider the following relations
		\begin{enumerate}
			\item $A(z|\theta^1,\theta^2,\theta^3,\theta^4)=(z+1|\theta^1,\theta^2,\theta^3,\theta^4)$,
			\item $B(z|\theta^1,\theta^2,\theta^3,\theta^4)=(z+\tau-2\theta^1\theta^2\theta^3\theta^4|\theta^1+\theta^2\theta^3\theta^4,\theta^2-\theta^1\theta^3\theta^4,\theta^3+\theta^1\theta^2\theta^4,\theta^4-\theta^1\theta^2\theta^3)$.
		\end{enumerate}
	Where $\tau$ is a non null even parameter. The quotient $\mathbb C^{1|4}/\langle A,B\rangle$ is an elliptic curve $\widetilde{\mathbb T}_{\tau}$ with a $SUSY_4$-super structure. Its tangent bundle $T_{\widetilde{\mathbb T}_\tau}$ does not have a subbundle of dimension $0|2$ then this curve does not come from a $1|2$-super curve.
\end{example}
	
\begin{observation}\label{dualidad}
	For $n=2$ the change of coordinates for our new variables $(\rho^1,\rho^2)$ are given by the equation:
		\begin{equation}\label{coordenadas susy 4}
				\left[ \begin{matrix}\eta^1 \\ \eta^2 \end{matrix} \right] = A^{-1} \left[ \begin{matrix}\partial_{1}F \\ \partial_{2}F\end{matrix} \right] + \Ber (\Phi) \left ( \left[ \begin{matrix}\partial_{2}\phi^2 & -\partial_{1}\phi^2 \\ -\partial_{2}\phi^1 & \partial_{1}\phi^1\end{matrix} \right] \left[ \begin{matrix}\rho^1 \\ \rho^2 \end{matrix}\right] + \left[ \begin{matrix}\partial_z \phi^2 \\ -\partial_z \phi^1\end{matrix} \right] \rho^1 \rho^2 \right ),
		\end{equation}	
	where $A=(\partial_{i}\phi^j )_{i,j}$ and this gives a structure of $SUSY_4$-super curve.
	
	More specifically, observe that the local coordinates $\rho^1\rho^2,\rho^1,\rho^2$ define a $1|2$ bundle over $C$ and since 
		\begin{equation*}
				\left[ \begin{matrix}\eta^1\eta^2 \\ \eta^2 \\-\eta^1 \end{matrix} \right] = \det A^{-1}\left[\begin{matrix} \partial_{1}F\partial_{2}F\\ \left[\begin{matrix}\partial_{2}\phi^1&\partial_{1}\phi^1 \\\partial_{2}\phi^2&\partial_{1}\phi^2 \end{matrix} \right] \left[ \begin{matrix}-\partial_{1}F \\ \partial_{2}F \end{matrix} \right] \end{matrix} \right] + \Ber (\Phi) J\Phi \left[ \begin{matrix}\rho^1\rho^2 \\ \rho^2\\-\rho^1 \end{matrix}\right].
		\end{equation*} 
	Then the sheaf $\mathcal{A}_{\widetilde{C}}$ that defines $\widetilde{C}$ is given by the extension of $\mathcal{O}_C$ by $\Ber_C\otimes \Omega_C$:
		\begin{equation*}
			0\to \mathcal{O}_C \to \mathcal{A}_{\widetilde{C}}\to \Ber_C\otimes \Omega_C\to 0.
		\end{equation*}
	When our curve is split, and associated to the bundle $E$, then $\widetilde{C}$ is a split curve associated the reduced curve $C_0$ and 
		\begin{equation*}
			W= E\oplus (E^*\otimes \Omega_{C_0}),
		\end{equation*}
	where $\Omega_{C_0}$ is the canonical bundle over $C_0$. In this special case, we can consider the \emph{dual curve} as the split curve associated to $E^*\otimes \Omega_{C_0}$.
\end{observation}

\section{$S(2)$-super curves and $SUSY_4$-super curves}
 \label{sec:duality}

In this section we consider the adjoint action $\Ad_\alpha:\Aut^\omega[[1|4]]\to\Aut^\omega[[1|4]]$ given by $\alpha\in\Aut^\omega[[1|4]]$:
	\begin{equation*}
		\alpha(z|\theta^1,\theta^2,\rho^1,\rho^2):=(z-\theta^1\rho^1-\theta^2\rho^2,\rho^1,\rho^2,\theta^1,\theta^2).
	\end{equation*}
Observe that $\alpha$ is an involution and $\Ad_\alpha$ fixes the subspace $\Aut^\omega[[1|2]]\subset\Aut ^\omega[[1|4]]$.

For the automorphism $\Ad_\alpha$ we have that $\Aut [[1|2]]\cap \Ad_{\alpha}\left(\Aut ^\delta[[1|2]]\right)=\Aut ^\Delta[[1|2]]$. Then we have the commutative diagram
	\begin{equation}\label{nomaestro}
		\begin{split}
			\xymatrixcolsep{4pc} \xymatrix{
				\Aut ^\Delta[[1|2]] \ar[d]^{j} \ar[r]^{\Ad_\alpha}& \Aut ^\Delta[[1|2]]\ar[d]^{j} \\
				\Aut ^\omega[[1|4]] \ar[r]^{\Ad_\alpha}& \Aut ^\omega[[1|4]]}
		\end{split}
	\end{equation}
Let $C$ be an $S(1|2)$-super curve, over the $\Aut^\delta[[1|2]]$-principal bundle, $\Aut^\delta_C$, construct the $\Aut^\omega[[1|4]]$-principal bundle, $\Aut^\omega[[1|4]]\times_{\Aut^\delta[[1|2]]}\Aut^\delta_C$, after this we obtain a morphism given by
	\begin{equation}\label{automorfismo}
		\begin{split}
			\mu:\Aut^\omega[[1|4]]\times_{\Aut ^\delta[[1|2]]}\Aut^\delta_C\to&\ \Aut^\omega[[1|4]]\times_{\Aut ^\delta[[1|2]]}\Aut^\delta_C\\
			(g,(x,\Phi))\mapsto&\ (\alpha\circ g,(x,\Phi)).
		\end{split}
	\end{equation}
	
If $C$ is an $S(2)$-super curve, then the bundle $K=\Aut^\delta_C/\Aut^\Delta\to C$ is trivial, we have a global section $s:C\to K$, defined by local sections $\{\phi_i\}_i$ of the bundle $\Aut^\delta_C\to C$ such that $\phi_i\circ\phi_j^{-1}\in\Aut^\Delta[[1|2]]$. Using \eqref{automorfismo}, we get local sections of $\Aut^\omega[[1|4]]\times_{\Aut ^\delta[[1|2]]}\Aut^\delta_C\to C$ by $\{\Ad_\alpha\left((1,\phi_i)\right)=(\alpha,\phi_i)\}_i$, since $\phi_i\circ\phi_j^{-1}\in\Aut^\Delta[[1|2]]$, then $\mu(\phi_i\circ\phi_j^{-1})\in\Aut^\Delta[[1|2]]$. In particular, such sections give us a global section of $\Aut^\Delta\backslash\Aut^\omega[[1|4]]\times_{\Aut ^\delta[[1|2]]}\Aut^\delta_C\to C$. Then, we obtain another family of $S(2)$-super curves $\widehat C\to S$. Such family is called dual family of curves, or simply dual curve.

The bundle $\Ad_\alpha(K)$ defines a family of $S(2)$-super curves if $C\to S$ is a family of $S(2)$-super curves.

\begin{observation}
	For a $\Phi\in\Aut ^\omega[[1|2]]$ we have $\mu(\Phi)=\Phi$, then for a $SUSY_2$-super curve in the diagram $\eqref{nomaestro}$ we have the dual curve is isomorphic to the original one.
\end{observation}

Similar to \eqref{fibrado trivial es S2} we obtain:

\begin{proposition}
	Given a family of $S(1|2)$-super curves $C\to S$, the image $\Ad_\alpha(K)$ in \eqref{automorfismo} defines a family of $1|2$-super curves if and only if $C\to S$ is a family of $S(2)$-super curves. In such case, we obtain the family of dual curves.
\end{proposition}

\begin{proof}
	The bundle $\Ad_\alpha(K)$ defines a $1|2$-super curve if and only if the projection of $\Ad_\alpha(K)$ over $\Aut[[1|2]]\backslash(\Aut^\omega[[1|4]]\times_{\Aut^\delta[[1|2]]}\Aut^\delta_C)$ has a global section, since $\Aut[[1|2]]\cap\mu\left(\Aut^\delta[[1|2]]\right)=\Aut^\Delta[[1|2]]$ then such projection is isomorphic to $K$, then $\Ad_\alpha(K)$ is trivial if and only if $K$ is trivial, that is if $C\to S$ is an $S(2)$-super curve.
\end{proof}

Then we obtain:

\begin{theorem}\label{teorema de automorfismo}
	There exists an involution $\mu$ of the moduli space $\cM_{S(2)}$ of $S(2)$-super curves. The fixed point set  of $\mu$ consists of the moduli space $\cM_{K(1|2)}$ of orientable $SUSY_2$-super curves. 
\end{theorem}

\begin{observation}
	This duality was observed in \cite{MasonTuiteYamskulna} as an involution of the super algebra $S(2)$.
\end{observation}

\begin{example}
	For a split $S(2)$ curve $C$ associated to $C_0$ and the vector bundle $E$, the dual curve $\widehat{C}$ is also split and is associated to $C_0$ and the vector bundle $E^*\otimes \Omega_{C_0}$.
\end{example}
	
\section{Families of super curves}
 \label{sec:families}
 In this section we construct examples of families of supercurves. We construct families of  $S(1|2)$ supercurves that are not $S(2)$ supercurves. Non-split $S(1|2)$ supercurves over a purely even base. Non-split $S(2)$ supercurves over a super-scheme. 

 \subsection{A family of $S(2)$-super curves}

\begin{example}\label{familia S(2) no split}
	Consider the $1|2$-dimensional family $\mathbb C^{1|3}\to\mathbb C^{0|1}=\spec(\mathbb C[\eta])$, with $\eta$ an odd variable. Over it we have the two automorphism given by
	\begin{equation*}
	\begin{split}
	A(t|\theta,\rho):=&\ (t+1|\theta,\rho),\\
	B(t|\theta,\rho):=&\ (t+\tau+\theta\eta|\theta,\rho),
	\end{split}
	\end{equation*}
	with $\tau\in\mathbb C$, $\Im(\tau)>0$. The quotient $\mathbb T_{\tau}:=\mathbb C^{2|3}\to\spec(\mathbb C[\eta])/\langle A,B\rangle$ is an analytical family of super torus. In order to see that this quotient is algebraic, let us recall the Weierstrass function $\wp$ given by the parameter $\tau$. Then we obtain the closed immersion:
	\begin{equation*}
	\begin{split}
	\mathbb T_{\tau}\to&\ \mathbb P^2(L)\\
	(t|\theta,\rho)\mapsto&\ (\wp(t;\tau+\theta\eta),\partial_t\wp(t;\tau+\theta\eta),1|\theta,\rho),
	\end{split}
	\end{equation*}
	where $E$ is a rank 2 trivial bundle over $\mathbb P^2$. The image of this immersion is given by the equation:
	\begin{equation}\label{equa}
	y^2=4x^3-g_2(\tau+\phi\eta)x-g_3(\tau+\phi\eta),
	\end{equation}
	with $(x,y,1|\phi,\psi)\in\mathbb P^{2}(L)$. Since \eqref{equa} is even, then $\mathbb T_\tau$ is a $1|2$-dimensional family.
	
	To see that this family is not split, suppose that there exists a $1|0$-dimensional family $M\to \spec(\mathbb C[\eta])$ and a rank 2 bundle $E_0$ over the family such that $M(E_0)=\mathbb T_\tau$. Observe that such family should be a family of torus, also the change of coordinates over any torus should have the form $\Phi(t|\theta,\rho)=(\phi(t)|\theta\lambda_{11}(t)+\rho\lambda_{12}(t),\theta\lambda_{21}(t)+\rho\lambda_{22}(t))$ and in this case $(\lambda_{ij}(t))$ corresponds to the cocycle of $E_0$. Over $\cT_{\mathbb T_\tau/S}$, $S=\spec(\mathbb C[\eta])$, we have the global section given by $\partial_t$, so we have the exact sequence
	\begin{equation*}
	0\to\langle\partial_t\rangle\to\cT_{\mathbb T_\tau/S}\to\cT_{\mathbb T_\tau/S}/\langle\partial_t\rangle\to 0.
	\end{equation*}
	Given a change of coordinates $\Phi(t|\theta,\rho)=(\phi(t)|\theta\lambda_{11}(t)+\rho\lambda_{12}(t),\theta\lambda_{21}(t)+\rho\lambda_{22}(t))$ then the change of coordinates of $\cT_{\mathbb T_\tau/S}/\langle\partial_t\rangle$ are given by $(\lambda_{ij}(t))$. On the other side, for coordinates $(t,\theta,\rho)$ the vector field $\partial_{\theta},\partial_\rho$ are well defined global section in $\cT_{\mathbb T_\tau/S}/\langle\partial_t\rangle$, that is $\cT_{\mathbb T_\tau/S}/\langle\partial_t\rangle$ is a trivial bundle, in particular $E_0$ is also a trivial bundle.  From this, the tangent bundle over $M(E_0)$ should be trivial, since the tangent bundle over the torus is trivial and we can define a global section $\partial_{\theta'}$ for any global section $\theta'$ in $E_0$. If this happens, then the space of global sections has dimension $1|2$. Let us take a section $s$ of $\cT_{\mathbb T_\tau/S}$. With respect to the et\'ale topology, from the projection $\mathbb T_{\tau}:=\mathbb C^{2|2}\to\spec(\mathbb C[\eta])\to\mathbb T_{\tau}:=\mathbb C^{2|3}\to\spec(\mathbb C[\eta])/\langle A,B\rangle$, we obtain a section of the tangent bundle $\cT_{\mathbb C/S}$, such section should have the form
		\begin{equation}\label{relaciones}
		\begin{split}
		s(t|\theta,\rho)=&\ s(t+1|\theta,\rho),\\
		s(t|\theta,\rho)=&\ s(t+\tau+\theta\eta|\theta,\rho).
		\end{split}
		\end{equation}
	Using the decomposition $s(t|\theta,\rho)=a(t|\theta,\rho)\partial_t+b(t|\theta,\rho)\partial_{\theta}+c(t|\theta,\rho)\partial_{\rho}$. From the relations \eqref{relaciones}, we obtain that $b$ should satisfy 
		\begin{equation}\label{relaciones1}
		\begin{split}
		b(t|\theta,\rho)=&\ b(t+1|\theta,\rho),\\
		b(t|\theta,\rho)=&\ b(t+\tau+\theta\eta|\theta,\rho).
		\end{split}
		\end{equation}
	from this $b$ should be constant. Analogously $c$ is constant. On the other side, we get:
		\begin{equation}\label{relaciones2}
		\begin{split}
		a(t|\theta,\rho)=&\ a(t+1|\theta,\rho),\\
		a(t|\theta,\rho)=&\ a(t+\tau+\theta\eta|\theta,\rho)-b\rho.
		\end{split}
		\end{equation}
	Similarly to \eqref{relaciones1}, taking derivative on \eqref{relaciones2} we obtain that $a$ is constant and $b=0$. That is, the vector space of sections has dimension $1|1$, and this contradicts that such space of sections has dimension $1|2$. Then the family of torus $\mathbb T_\tau$ is not split.	
	
	This family has a global Berezinian given by $\Delta=[dz|d\theta d\rho]\in H^0(C,\Ber_{C/S})$. Since it has a dual curve given by the quotient of $\mathbb C^{1|3}\to\spec(\mathbb C[\eta])$ by the automorphisms:
		\begin{equation*}
		\begin{split}
		\widetilde A(t|\theta,\rho):=&\ (t+1|\theta,\rho),\\
		\widetilde B(t|\theta,\rho):=&\ (t+\tau|\theta-\eta,\rho).
		\end{split}
		\end{equation*}
	The quotient $\widehat{\mathbb T}_{\tau}:=\mathbb C^{1|3}/\langle A,B\rangle\to\spec(\mathbb C[\eta])$ is the dual curve. Then $\mathbb T_{\tau}$ is a family of $S(2)$-super curves that is no split.
	
	Also, observe that this family does not have a structure of $SUSY_2$-super curve. Since, if this family has a rank $0|2$ distribution on $\cT_{\mathbb T_\tau/S}$.
\end{example}

	\subsection{A family of $S(1|2)$-super curves}
	\begin{example}Over the affine plane $\mathbb C^{1|2}$ consider the following relations
		\begin{enumerate}
			\item $A(z|\theta^1,\theta^2):=(z+1|\theta^1,\theta^2)$,
			\item $B(z|\theta^1,\theta^2):=(z+\tau+\theta^1\theta^2|\theta^1,\theta^2)$,
		\end{enumerate}
		where $\tau$ is even. The quotient $\mathbb C^{1|2}/\langle A,B\rangle$ is an elliptic curve $\mathbb T_{\tau}$ with Berezinian $[dz|d\theta^1d\theta^2]$. Consider the $SUSY_4$-super curve associated, $\widetilde{\mathbb T}_{\tau}$, since $\cT_{\widetilde{\mathbb T}_{\tau,\rho}}$ does not have a splitting. Then $\mathbb T_{\tau}$ does not have an $S(2)$-structure.
\label{ex:non-split12}
	\end{example}

\section{On the moduli space of curves with a trivial Berezinian}
 \label{sec:moduli}

In \cite{Vaintrob}, Vaintrob studied the moduli space of super curves with a fixed Berezinian. Given a family $X\to S$ of such supercurves, with special fiber $X_0$ over $s_0 \in S$, Vaintrob views the generic fiber as a deformation of the curve $X_0$ together with a deformation of a section $\Delta\in H^0(X_0,\Ber_{X_0})$. 

We will follow the recipe given in \cite{lebrun1988} to study the moduli space of $S(2)$-super curves. If $S$ is a scheme, or a purely-even super-scheme, we will say that such deformations are \emph{even}. For a super scheme $S$, if we fix a family $C_0\to S_\rd$ any extension $C\to S$ is called an \emph{odd} deformation of $C_0\to S_\rd$.

We divide the process in three steps:

\begin{enumerate}
	\item We study the restriction of the functor to schemes. Suppose that such functor is represented by the scheme $A$ and comes with a universal curve $\cA\to A$.
	\item We focus on odd deformations of such curves. Suppose that such deformation is given by a functor $S\to E(S)$, where $S$ is a scheme and $E(S)$ is a set.
	\item We check what kind of automorphisms has our candidate to the moduli space.
\end{enumerate}

The construction is divided in these steps inspired by the following lemmas:

\begin{lemma}
	Let $M$ be a super scheme, that defines the functor $h_M:SSch\to Sets$, $h_M(N)=Hom_{SSch}(N,M)$, then the restriction $h_M|_{Sch}$ is represented by $M_\rd$.
\end{lemma}

\begin{proof}
	This follows directly from Lemma \ref{caracterizacion de reduccion}.
\end{proof}

\begin{lemma}
	Let $M$ be a split super scheme, defined by the scheme $j:M_\rd\hookrightarrow M$ and a locally free $\cO_{M_\rd}$-module $E$, then $E^*\simeq \ker(j^*\Omega_M\to\Omega_{M_\rd})$.

\end{lemma}

\begin{proof}
	This follows directly from Observation \ref{variedades split}.
\end{proof}

That is, if our moduli space is given by a locally split super scheme, then such space should be described by the reduced space representing even deformations and the fiber bundle defined (locally) by the purely odd deformations with a fixed even base. To finish this process, we find the automorphisms that fix such families of curves.

\begin{observation}
	The process described above is justified by the following construction: Suppose that we have a family $X\to S$ of $S(2)$ curves, then the diagram
		\begin{equation}\label{familia reducida}\xymatrix{
			X_0:=X\times_{S_\rd}S\ar[r]\ar[d]&X\ar[d]\\
			S_\rd\ar[r]&S}
		\end{equation}
	Since $S_\rd$ is a scheme and $X_0\to S_\rd$ is family of $S(2)$ super curves, then this family is given by a morphism $\phi_0:S_\rd\to A$ and the pullback $\phi_0^*\cA\to S_\rd$. Since locally, $S$ is a split scheme, then in an open neighbourhood $U\subset S_\rd$ there exists a fiber bundle $E$ such that, locally, $S|_U=U(E)$. Assuming that $U$ is affine we will proof uniqueness on the extension of such family $X_0(U)\to U$. 
	
	Finally the family $\pi:X\to S$ is going to be described as the gluing of local open pieces $\{\pi^{-1}(U_i)\to U_i\}_i$ through morphisms $U_i\to A$. Finally, the inner automorphism is going to give us an orbifold description of such object.
\end{observation}

\subsection{The even part}

Observe that for an $S(2)$-super curve, over its $SUSY_4$-super curve $\widetilde{C}$ we have the local vector fields $ D^i_+=\partial_{\rho^i}$ and $ D^j_-=\rho^j\partial_{z}+\partial_{\theta^j}$ on the coordinate charts $(z|\theta^1,\theta^2,\rho^1,\rho^2)$ given by \eqref{coordenadas n a Susy 2n} satisfying the relation $\{D^i_+,D^j_- \}=\delta_{ij}\partial_z$. The change of coordinates satisfies equations \eqref{ecuacion de otras coordenadas de N=k SUSY curvas}:
\begin{equation*}
D^j_-F=\eta^1D^j_-\phi^1+\eta^2D^j_-\phi^2,
\end{equation*}
thus we have
\begin{equation*}
\{D^i_+,D^j_-\}F=D^i_+\eta^1D^j_-\phi^1-\eta^1\{D^i_+,D^j_-\}\phi^1+D^i_+\eta^2D^j_-\phi^2-\eta^2\{D^i_+,D^j_-\}\phi^2,
\end{equation*}
and that implies:
\begin{equation*}
\delta_{ij}(\partial_zF+\eta^1\partial_z\phi^1+\eta^2\partial_z\phi^2) = D^i_+\eta^1D^j_-\phi^1+D^i_+\eta^2D^j_-\phi^2.
\end{equation*}

It follows from \eqref{coordenadas susy 4} that
\begin{equation*}
\left[\begin{matrix} D^1_+\eta^1 & D^1_+\eta^2\\
D^2_+\eta^1 & D^2_+\eta^2 \end{matrix}\right] = \left[\begin{matrix} D^2_-\phi^2 & - D^2_-\phi^1\\
- D^1_-\phi^2 & D^1_-\phi^1 \end{matrix}\right].
\end{equation*}
Finally, we get
\begin{equation}\label{determinant}
\partial_zF+\eta^1\partial_z\phi^1+\eta^2\partial_z\phi^2= D^1_-\phi^1 D^2_-\phi^2- D^1_-\phi^2 D^2_-\phi^1,
\end{equation}
then we have that the projection:
\begin{equation*}
\det E\to\Omega_C,
\end{equation*}
is an isomorphism. 

Conversely, suppose that we start with a $1|0$-family over a purely even base $S$, $\pi_0:C_0\to S$,  and a rank $2|0$ bundle $E$ with an isomorphism $\beta:\det E\to \Omega^1_{C_0}$. Given such data, we construct $\pi:C\to S$ given by
\begin{equation*}
\begin{split}
C_{\text{rd}}=&\ C_{0},\\
\mathcal{O}_{C,0}=&\ \mathcal{O}_{C_0}\oplus\det E,\\
\mathcal{O}_{C,1}=&\ \Pi E.
\end{split}
\end{equation*}

The $S(2)$ structure is given (locally) by the coordinates $z$ and a local frame $\theta^1,\theta^2$ of $E$ such that $\beta(\theta^1\otimes\theta^2)=dz$. This data defines the local section of the Berezinian $[dz|d\theta^1d\theta^2]$. Since $\beta$ is an isomorphism, then such class is well defined. Also, observe that for any change of coordinates compatible with the Berezinian, $\Phi=(F|\phi^1,\phi^2)$, where $F(z|\theta^1,\theta^2)=F(z)$, the decomposition (in Proposition \ref{descomposicion 0}), $\Phi=\exp(X)\circ\phi$, gives us that $X\in S(2)$.

Similar to \cite{manin2014topics} we obtain:

\begin{proposition}\label{even data}
	For any family $\pi_0:C_0\to S$ of relative dimension $1|0$ over a purely even base, the following data are equivalent:
	\begin{enumerate}
		\item An $S(2)$-family of curves $\pi:C\to S$ over $S$ with $C_{\text{rd}}=C_{0,\text{rd}}$ and $\mathcal{O}_{C,0}=\mathcal{O}_{C_0}(1)$, where $C_0(1)$ is the first neighborhood of the diagonal in $C_0\times_S C_0$.
		\item A rank $2$ bundle $E$ and an isomorphism $\det E\overset{\beta}{\to} \Omega_{C_0/S}$ up to equivalence: a pair $(E,\beta)$ is equivalent to $(E',\beta')$ is there exists an isomorphism, such that the following diagram commutes
		\begin{equation*}
		\xymatrix{\det E\ar[r]\ar[rd]_{\beta} & \det E'\ar[d]^{\beta'}\\
			& \Omega_{C_0}}.
		\end{equation*}
	\end{enumerate}
\end{proposition}

\begin{proof}
	The previous comment shows $(2)\to(1)$.
	
	To see $(1)\to (2)$ consider the $SUSY_4$-super curve $\widetilde{C}$ associated to $C$. Since $C$ is $S(2)$, we have the well defined $0|2$ bundle
	\begin{equation*}
	\widehat{E}=\langle\rho^1\partial_z+\partial_{1},\rho^2\partial_z +\partial_{2}\rangle
	\end{equation*}
	over $\widetilde{C}$. The pullback $E:=i^*(\Pi\widehat E)$ over the inclusion $i:C_{\text{rd}}\hookrightarrow \widetilde{C}$ is a rank $2$ bundle. We get from equation \eqref{determinant} that $\det E\simeq \Omega_{C_0}$, where $C_0=C_{\text{rd}}$. Finally, since $C$ is $S(2)$, then $\mathcal{O}_{C,0}=\mathcal{O}_{C_0}\oplus\det E$ and $\mathcal{O}_{C,1}=\Pi E$.
\end{proof}

In this case, we have that such moduli space is parametrized by the space of curves, $C_0$, and a rank 2 fibre bundles $E$ with a marked isomorphism $\det E\simeq\Omega_{C_0}$.

\begin{observation}
	To calculate the dimension of the moduli space over a point $[(C,\Delta)]$ consider a deformation given by the class $\{X_{ij}\}\in H^1(C,\cT_{C})$ with respect to a covering $\{U_i\}_i$. If we want that the perturbation keep in the space of curves with a trivial Berezinian then we need that the class of $\{\sdiv_{\Delta}(X_{ij}) \}\in H^1(C,\mathcal O_{C})$ to be zero. So, the space of variations of curves with a trivial Berezinian is given by the kernel of the homomorphism of group
		\begin{equation*}
			\begin{split}
				H^1( C,\cT_{ C})\stackrel{\phi}{\to} &\ H^1( C,\mathcal O_{ C})\\
				\{X_{ij}\} \mapsto &\ \{\sdiv_{\Delta}(X_{ij}) \}
			\end{split}	
		\end{equation*}
	Similarly, if we want deformations of $S(2)$-super curves also we need that such perturbation is in the kernel of
		\begin{equation*}
			\begin{split}
				\ker(\phi)\stackrel{\psi}{\to} &\ H^1\left(C,\Aut^\Delta[[1|2]]\backslash\Aut^\delta_C\right)\\
				A=\{X_{ij}\}  \mapsto &\ \Gamma_A,
			\end{split}	
		\end{equation*}
	recall that $\Gamma_A=\{\exp(X_{ij})\}_{ij}$.
	
	Finally, the dimension of the moduli space of $S(2)$-super curves over $[\mathcal C]$ is given by $\ker(\psi)$ and $H^0(\mathcal C,\text{Ber}_{\mathcal C})$.
\end{observation}

Given a genus $g$, and a degree $d=g-1$. For the moduli space of genus $g$ curves with a rank 2 vector bundle with a fixed degree $d$, denoted by $\cM_{g,2,d}$ we consider the morphism to the space of curves with a line bundle with a fixed degree $2d$, $P_{g,2d}$, given by:
	\begin{equation*}
	\begin{split}
		\cM_{g,2,d}\to&P_{g,2d}\\
		[E\to C]\mapsto&[\det E\to C].
	\end{split}
	\end{equation*}
We obtain that the preimage of $\{[\Omega_C\to C]\}$, denoted by $M_{g,2,\Omega}$, is the space of curves with genus $g$ and a rank 2 bundle $E$ together with an isomorphism $\det E\to \Omega_C$. Recall that the smooth part is given by the stable bundles.

The universal curve is given by the projection $X(E)\mapsto [E\to X]$.

\begin{observation}
	The involution described in Section \ref{sec:duality} is exactly the same as the one mentioned in Observation \ref{dualidad}.
\end{observation}

\subsection{The odd part}

Let $\pi:C\to S$ be a family of $S(2)$super curves with a fixed genus $g$, for a split supermanifold $S$, with reduced space $S_\rd$ and $\cO_{S_\rd}$-free module $W$. Considering the closed point $s_0\in S$ and the fiber $C_{s_0}\to\{s_0\}$, there exists a covering by affine open sets $\{U_i\}$ of $C_{s_0}$ such that the family $C\to S$ restricts to $U_i\times S\to S$. This observation follows from \cite[Theorem~1.2.4]{sernesi2007deformations}. Then the global family is determined by cocycle given by the change of coordinates in $\Phi_{ij}\in \Aut((U_i\cap U_j)\times S)$.

Considering the reduced family $C_0\to S_\rd$ given by \eqref{familia reducida}, fixing the reduced family, then the reduced part of the change of coordinates $\Phi_{ij}$ is fixed, denote it by $\phi_{ij}\in\Aut((U_i\cap U_j)\times S)_\rd$. Then, the $\phi_{ij}^{-1}\Phi_{ij}$ is an automorphism being the identity over the reduced space, then this morphism is the exponential of a nilpotent vector field relative to $S$. To get a family of $S(2)$ curves we need the distribution:
	\begin{equation*}
		\mathcal D:=S(2)\cap \cT_{C/S}
	\end{equation*}
where $S(2)$ was described in \eqref{fibrado S(2)}, and the sheaf $\mathcal N$ of nilpotent elements of $\Lambda^\bullet\pi^* W$. Then, $\phi_{ij}^{-1}\Phi_{ij}=\exp(X_{ij})$ for an even vector field $X_{ij}\in \mathcal N\otimes\mathcal D$. The group bundle 
	\begin{equation}\label{group bundle}
		G=\exp((\mathcal N\otimes\mathcal D)_0),
	\end{equation}
and the deformations are parametrized by $H^1(C_0,G)$. 

Finally, we get

\begin{proposition}\label{odd part}
	Let $S$ be a split super scheme and a family $C\to S$ of $S(2)$-super curves, then this family correspond to a extension of $C_0\to S_\rd$ and a class in $H^1(C_0,G)$, for $G$ in \eqref{group bundle}.
\end{proposition}

\subsection{Inner Automorphism}

Now, we will check which automorphisms preserves the family of curves. We are looking for maps $\psi:S\to S$ such that the families of $S(2)$ super curves $C\to S$, $\psi^*C\to S$ are equal, in this case we are going to say that $\psi:S\to S$ preserve the family. First consider the following lemma:

\begin{lemma}
	Let $\pi:C\to S$ be a family of curves, suppose that $\psi:S\to S$ preserve the family and $\psi_\rd:S_\rd\to S_\rd$ is the identity, then $\psi$ is the identity.
\end{lemma}
\begin{proof}
	Since, we can cover $S$ by open split schemes $\{S_i\}_i$, then $\psi|_{C_i}:C_i\to S_i$, with $C_i=\pi^{-1}(S_i)$, is the identity over $S_{i,\rd}$. The family $C_i\to S_i$ is defined by $C_{i,0}\to S_{i,\rd}$ and a class $H^1(C_{i,0},G_i)$. Since $\psi$ define the same family, then both classes should coincide, then $\psi$ should be also the identity.
\end{proof}

From the previous lemma, it follows that any automorphism of the family $X\to S$ is given by an automorphism over the reduced space $S_\rd$.

Let $\pi:C\to S$ be a family of $S(2)$-super curves, $S$ is a scheme and $E$ be the rank 2 bundle defining $\pi$. Considering the reduction $C_\rd\to S$, we obtain a family of genus $g$ curves over $S$. From now on we consider only the case $g\geq 2$. There does not exists any automorphism different from the identity. Then, in order to study such automorphism we have to check what happens in the odd part generated by the bundle $E$, that is we have to study the automorphisms of the rank $2$ bundle $E\to C_\rd$, with the chosen isomorphism $\det E\to C_\rd$. Considering the stable bundles, we have $\Aut_{C_\rd}(E)=\pi^*\cO_S^*$, since we need that such automorphism preserve the isomorphism $\det E\to C_\rd$, this group reduces to $\pm1$. The corresponding automorphism given by $1$ is the identity, while the automorphism given by $-1$ is denoted by $\Psi$. Then we get the lemma:

\begin{lemma}
	Let $\pi:C\to S$ be a family of genus $g$ $S(2)$-super curves, $S$ is a scheme and $E$ be the rank 2 stable bundle defining $\pi$. Then any automorphism $\psi:S\to S$ preserving the family corresponds to a class $\sigma_\psi\in H^1(S,\mathbb Z/2\mathbb Z)$.
\end{lemma}

Finally, we get

\begin{proposition}
	Let $\pi:C\to S$ be a family of genus $g$ $S(2)$-super curves. Then any automorphism $\psi:S\to S$ preserving the family corresponds to a class $\sigma_\psi\in H^1(S_\rd,\mathbb Z/2\mathbb Z)$.
\end{proposition}

\begin{theorem}
	The data $(M_{g,2,\Omega}, \{U_i\}_i,\delta_{ij})$, where $M_{g,2,\Omega}$ is the moduli space of genus $g$ curves $C$ joint with a rank 2 bundle $E\to C$ with a fixed isomorphism $\det E\simeq \Omega_C$, $U_i$ is a split super scheme such that $\cup U_{i,\rd}=M_{g,2,\Omega}$ and the associated bundle is given in Proposition \ref{odd part} considering the universal family $C_{U_{i,\rd}}\to U_{i,\rd}$, and $\delta_{ij}:U_i|_{U_{ij,\rd}}\to U_j|_{U_{ij,\rd}}$ isomorphism such that $\delta_{ij}\circ\delta_{jk}\circ\delta_{ki}\in\{\text{id},\Psi\}$; models the orbifold representing the functor of $S(2)$ super curves.
\end{theorem}

\subsection{The even data for $S(1|2)$-super curves}

A general family of $S(1|2)$-super curve $C\to S$ over the even base $S$, with a nonvanishing section $\Delta\in H^0(C,\Ber_{C/S})$, defines a class 
\begin{equation}\label{igualdad}
\Gamma_C\in H^1(C,\pi^*\mathcal O_S),
\end{equation}
given by the bundle $K$ in \eqref{fibrado}. Suppose that this class is defined by a covering $\{U_i\}$ with local coordinates $\Phi_i$ compatible to $\Delta$ and $\Gamma_C=\{\lambda_{ij}\}$, then the change of coordinates is given by
\begin{equation*}
\begin{split}
z_j=&\ F_{ij}(z_i)+\lambda_{ij}\partial_{z_i}F_{ij}(z_i)\theta^1_i\theta^2_i\\
\theta^1_j=&\ \theta^1_ia_{11}(z_i)+\theta^2_ia_{12}(z_i)\\
\theta^2_j=&\ \theta^1_ia_{21}(z_i)+\theta^2_ia_{22}(z_i).
\end{split}
\end{equation*}
Here the covering $\{U_i\}$ and the change of coordinates $\{F_{ij}\}$ defines a curve $C_0$ and
\begin{equation*}
A_{ij}=\left[\begin{matrix}a_{11}(z_i)&a_{12}(z_i)\\
a_{21}(z_i)&a_{22}(z_i)\end{matrix}\right]
\end{equation*}
defines a rank 2 bundle $E$ over $C_0$ with $\det E\simeq\omega_{C_0}$.
We can define the new curve $C'=C_0(E)$ that is actually an $S(2)$-super curve.

The sheaf $\mathcal{O}_{C,0}$ defined over $C_0$ is a sheaf of algebras that is isomorphic to $C_0(1)$, the first neighbourhood of the diagonal on $C_0\times_SC_0$, if and only if $\{\lambda_{ij}\}\in H^1(C_0,\pi^*\mathcal{O}_S)$ vanishes. In general, each class $\Gamma\in H^1(C_0,\pi^*\mathcal{O}_S)$ defines a curve $C_{\Gamma}=(C_0,\mathcal{O}_{C,0})$ by the diagram from \eqref{Fig1}. First take $\mathcal L=L\otimes \mathcal O_{C_0}$ and $\mathcal A=\alpha(\mathcal O_{C_0})$, with these we construct the sheaf
\begin{equation*}
\begin{split}
\mathcal{O}_{C,0}:=\frac{T^{\bullet\geq1}\mathcal L}{\langle x\otimes y-y\otimes x,a\otimes a:x,y\in\mathcal L, a\in\mathcal A \rangle}.
\end{split}
\end{equation*}
Observe that this sheaf fits in the square-zero extension of algebras:
\begin{equation*}
\begin{split}
0\to\Omega_{ C_0}\to\mathcal{O}_{C,0}\to\mathcal{O}_{C_0}\to 0,
\end{split}
\end{equation*}
Since we consider terms of degree greater then zero, then $\mathcal{O}_{C,0}$ does not necessarily have the structure of an $O_{C_0}$-algebra.

Finally, we consider 
\begin{equation*}
\begin{split}
\mathcal{O}_{C,1}:=\Pi E,
\end{split}
\end{equation*}
with multiplication given by the isomorphism
\begin{equation*}
\begin{split}
\det E\to\Omega_{ C_0}. 
\end{split}
\end{equation*}
So, we get the $S(2)$-super curve $(C,\mathcal O_C)$. Finally:

\begin{proposition}\label{even data S(1|2)}
	For any family over a purely even base $\pi_0:C_0\to S$ of relative dimension $1|0$, the following data are equivalent:
	\begin{enumerate}
		\item An $S(1|2)$-family of curves $\pi:C\to S$ over $S$ with $C_{\text{rd}}=C_{0,\text{rd}}$ and $\mathcal{O}_{C,0}=\mathcal{O}_{C_\Gamma}$.
		\item A class $\Gamma\in H^1(C_0,\pi^*\mathcal{O}_S)$ and a rank $(2|0)$ bundle $E$ with an isomorphism $\det E\overset{\beta}{\to} \Omega_{C_0/S}$ up to equivalence: a triple $(E,\beta,\Gamma)$ is equivalent to $(E',\beta',\Gamma')$ if there exists an isomorphism $\alpha: E \rightarrow E'$, such that $\Gamma = \Gamma'$ and the diagram 
		\begin{equation*}
		\xymatrix{\det E\ar[r]\ar[rd]_{\beta} & \det E'\ar[d]^{\beta'}\\
			& \Omega_{C_0/S}}
		\end{equation*}
		commutes.
	\end{enumerate}
\end{proposition}

\begin{proof}
	The previous comment shows $(2)\to(1)$.
	
	To see $(1)\to (2)$ we get the class $\Gamma\in  H^1(C_0,\pi^*\mathcal{O}_S)$ by taking $\lambda_{ij}=\gamma(\Phi_{ij})$ as in \eqref{igualdad}, and considering the $S(2)$-super curve $C'=C(E)$, from the comment, the associated $C'$ super curve verifies what we want.
\end{proof}

\subsection{The genus 1 curve}

For the special case of genus 1 curves, we have that the even part is given by an ordinary curve $E_0$ and an element of $H^1(E_0,SL(2,\mathbb C))$.

Is known that the space of elliptic curves is given by a quotient of the upper half space $\mathbb H$ by the group $SL(2,\mathbb Z)$, where any element $\tau\in\mathbb H$ defines a quotient of $\mathbb C$ by the action of the group $\mathbb Z_\tau=\{a+b\tau:a,b\in\mathbb Z\}$. Similarly, any element $(\tau,a)\in\mathbb H\times \mathfrak{sl}(2,\mathbb C)$, the quotient of $\mathbb C^{1|2}$ by the action of $A(z|\theta^1,\theta^2)=(z+1|\theta^1,\theta^2)$ and $B(z|\theta^1,\theta^2)=(z+\tau|(\theta^1,\theta^2)\exp(2\pi ia))$.
	
A family of even deformation is given by $\mathbb H\times \mathfrak{sl}(2,\mathbb C)$ considering the action of the following three groups: $SL(2,\mathbb Z)$, $\mathbb Z^2$, $SL(2,\mathbb C)$.
	\begin{enumerate}
		\item The group $SL(2,\mathbb Z)$: Consider the action 
			\begin{equation*}
				\left(\begin{matrix}a&b\\c&d\end{matrix}\right)\cdot (\tau,a) =\left(\frac{a\tau+b}{c\tau+d},\frac{a}{c\tau+d}\right)
			\end{equation*}
		and for the quotient, we have the isomorphism induced by
			\begin{equation*}
				\left(\begin{matrix}a&b\\c&d\end{matrix}\right)\cdot (z|\theta^1,\theta^2) =\left(\frac{z}{c\tau+d}\Bigg |(\theta^1,\theta^2)\exp\left(-2\pi iz\frac{ca}{c\tau+d}\right)\right)
			\end{equation*}
		and observe that such isomorphism preserves the Berezinian if and only if $c=0$ and $d=1$.
		\item The group $\mathbb Z^2$: Consider the action 
			\begin{equation*}
				(m,n)\cdot (\tau,a) =(\tau,a+m\tau+n)
			\end{equation*}
		and for the quotient, we have the isomorphism induced by
			\begin{equation*}
				(m,n)\cdot (z|\theta^1,\theta^2) =(z |(\theta^1,\theta^2)\exp\left(2m\pi iz\right))
			\end{equation*}
		and observe that such isomorphism preserves the Berezinian if and only if $m=0$.
		\item The group $SL(2,\mathbb C)$: Consider the action 
			\begin{equation*}
				C\cdot (\tau,a) =(\tau,CaC^{-1})
			\end{equation*}
		and for the quotient, we have the isomorphism induced by
			\begin{equation*}
				C\cdot (z|\theta^1,\theta^2)=(z|(\theta^1,\theta^2)C)
			\end{equation*}
		and observe that such isomorphism preserves the Berezinian.
	\end{enumerate}
Finally, if we consider the group $SL(2,\mathbb Z)\ltimes\mathbb Z^2\times SL(2,\mathbb C)$ with the product
	\begin{equation*}
		(A,\gamma,C)\cdot (A',\gamma',C')=(AA',\gamma'+\gamma A',CC')
	\end{equation*}
we get the fine moduli space $\mathcal{M}_0=\mathbb H\times \mathfrak{sl}(2,\mathbb C)\sslash SL(2,\mathbb Z)\ltimes\mathbb Z^2\times SL(2,\mathbb C)$ of the even families of $S(2)$-super curves.

\begin{observation}The moduli space of orientable $SUSY_2$-super curves is given by $(\tau,a)$, where $a\in \mathfrak{sl}(2,\mathbb C)$ is diagonal. In this case and coordinates given by $(z|\theta^1,\theta^2)$, the $SUSY$-structure is obtained by $\omega=dz+\theta^2d\theta^1$.
\end{observation}
	
	\nocite{*}
\bibliography{bibfile2}

\end{document}